\documentclass[a4paper,12pt]{article} 

\usepackage{amsmath,amssymb,amsthm}
\usepackage[usenames]{color}
\usepackage{graphicx}
\usepackage{amsfonts}
\usepackage{comment}

\newtheorem{thm}{Theorem}[section]
\newtheorem{prop}[thm]{Proposition}
\newtheorem{lmm}[thm]{Lemma}
\newtheorem{defn}[thm]{Definition}
\newtheorem{remark}[thm]{Remark}

\newtheorem{cry}[thm]{Corollary}

\newtheorem{ntn}{Notation}[section]

\numberwithin{equation}{section}

\newcommand{\pr}{\noindent{\bf [Proof]} \qquad}
\newcommand{\prend}{\hfill \qed \medskip}

\newcommand{\Ker}{{\rm Ker}}
\newcommand{\Comm}{{\rm Comm}}
\newcommand{\LCM}{{\rm LCM}}
\newcommand{\rank}{{\rm rank}} 
\newcommand{\1}{{\bf 1}} 

\newcommand{\wt}{{\rm wt}}
\newcommand{\Aut}{{\rm Aut}}

\newcommand{\CF}{{\cal F}} 
 
\newcommand{\CG}{{\cal G}} 
\newcommand{\CH}{{\cal H}}

\newcommand{\bC}{{\mathbb C}} 
\newcommand{\bZ}{{\mathbb Z}} 
\newcommand{\bN}{{\mathbb N}} 
\newcommand{\bQ}{{\mathbb Q}} 
 
\newcommand{\bR}{{\mathbb R}}

\definecolor{dark}{rgb}{0.1,0.1,0.1}
\definecolor{skyblue}{rgb}{0.5,0.5,1}
\definecolor{fadegreen}{rgb}{0.9,1,0.9}

\begin{document}

\title{Orbifold construction and Lorentzian construction of Leech lattice vertex operator algebra}

\author{
\begin{tabular}{c}
Naoki Chigira$^{(1)}$, 
Ching Hung Lam$^{(2)}$
\footnote{Partially supported by grant AS-IA-107-M02 of Academia Sinica  and MoST grant  107-2115-M-001 -003-MY3  of Taiwan}  \
and Masahiko Miyamoto$^{(3)}$
\footnote{Partially supported 
by the Grants-in-Aids
for Scientific Research, No.18K18708, The Ministry of Education,
Science and Culture, Japan.}
\cr 
{\small $(1)$ Department of Mathematics, Kumamoto University} \cr
{\small $(2),(3)$ Institute of Mathematics, Academia Sinica } \cr
{\small $(3)$ Institute of Mathematics, University of Tsukuba }
\end{tabular}
}
\date{}

\maketitle

\begin{abstract}
We generalize Conway-Sloane's constructions of the Leech lattice from Niemeier lattices
using Lorentzian lattice to holomorphic vertex operator algebras (VOA) of central charge $24$. It provides a tool for analyzing the structures and relations among holomorphic VOAs related by orbifold construction. In particular, we are able to get some  useful information about certain lattice subVOAs associated with the Cartan subalgebra of the weight one Lie algebra $V_1$. We also obtain  a relatively elementary proof that any strongly regular holomorphic VOA of central charge $24$ with $V_1\neq 0$ can be constructed directly by a single orbifold construction from the Leech lattice VOA without using a dimension formula. 
\end{abstract}

\section{Introduction}
As a recent development in the research of vertex operator algebras (shortly VOAs), the classification 
of strongly regular holomorphic vertex operator algebras of central 
charge $24$ is accomplished except for the uniqueness of holomorphic VOAs of moonshine type. 
As a result, there exist exactly 70 strongly regular holomorphic VOAs $V$ 
of CFT-type with central charge $24$ and non-zero weight one space 
$V_1\not=0$  
and the possible Lie algebra structures for their weight one subspaces have been given 
in Schellekens' list \cite{Schel}. 
Since their constructions and the uniqueness proofs were  done case by case,  a simplified 
construction and proof of uniqueness is somehow expected. 
As an answer to this problem, Scheithauer and M\"{o}ller \cite{MoScheit} have 
recently announced the existence 
of direct orbifold constructions from VOAs in Schellekens' list to 
the Leech lattice VOA $V_{\Lambda}$ by using their dimension formula. 
As a consequence, there are also direct orbifold constructions from 
$V_{\Lambda}$ to the $69$ holomorphic VOAs with semisimple weight one Lie algebras by using the reverse automorphisms. They called these $69$ (isomorphism classes of) automorphisms ``generalized deep holes 
of Leech lattice".  
In \cite{ELMS}, another proof that any holomorphic VOA of central charge $24$ with $V_1\neq 0$ can be constructed by a single orbifold construction from the Leech lattice is also given. Moreover, a relatively simpler proof for the Schellekens'
 list is obtained.

The main purpose in this paper is to try to understand the relations among holomorphic VOAs of central charge $24$. Although we can construct holomorphic VOAs from the Leech lattice VOAs directly by 
orbifold constructions, 
it is usually not easy to analyze their structures except for those obtained by inner automorphisms. Therefore, it may still be worthwhile to study relations among holomorphic VOAs before we arrive at the Leech lattice VOAs from holomorphic VOAs by 
inner automorphisms. 
The theory of strongly regular VOA (or vertex algebra) shares several 
parallel properties with lattice theory. For example, if we have an even lattice $E$, 
we can construct a vertex algebra $V_E$ called a lattice vertex 
algebra (abbreviated as VA), (see \cite{FLM}). 
A classification of unimodular positive definite even lattices 
has been given by Niemeier \cite{Nie} and then  Conway and Sloane \cite{ConSl} later
gave a unified construction from Leech lattice to the other 23 Niemeier lattices.  
In this paper, following their method using Lorentzian lattice and 
isotropic elements, 
we will give a relatively elementary proof for the existence of generalized deep holes 
without using Schellekens' list nor the dimension formula, 
but with the help of the knowledge about the Conway group $Co_0$, the isometry group of the Leech lattice $\Lambda$. Especially, 
we will get useful information about some lattice subVOAs associated with the Cartan subalgebra of the weight one Lie algebra $V_1$. 
In 2004, Dong and Mason \cite{DM2} showed that if $V$ is a strongly 
regular holomorphic VOA of CFT-type with central charge $24$ and $V_1\not=0$ 
and $V\not\cong V_{\Lambda}$, then $V_1$ is a semisimple Lie algebra, say 
$V_1=\oplus_{j=1}^t \CG_j$. 
\medskip

As one of our results, we will prove the following theorem using Lorentzian lattice and some arguments in lattice theory. 

\begin{thm}\label{Main}
Let $V$ be a strongly regular holomorphic VOA of CFT-type with central charge $24$. Assume that $V_1\cong \oplus_{j=1}^t \CG_j \not=0$ is semisimple, where $\CG_j, j=1, \dots,t,$ are  simple Lie algebras. Let $\alpha=\sum_{j=1}^t \rho_j/h_j^{\vee}\in V_1$  be the sum of Weyl vectors $\rho_j$ of $\CG_j$ divided by the dual Coxeter numbers $h_j^{\vee}$. We call $\alpha$ W-element. 
Then the orbifold construction $\tilde{V}(g)$ from $V$ by using an inner automorphism $g=\exp(2\pi i\alpha(0))$  gives the Leech lattice VOA $V_{\Lambda}$, where $Y(\alpha,z)=\sum \alpha(n)z^{-n-1}$ is a vertex operator of $\alpha\in V$. 
\end{thm}

Although this theorem has already been proved  in \cite{ELMS} and \cite{MoScheit} by using dimensional formula,  
our main aim is to analyze structural relations among holomorphic VOAs 
in the process of the proof. 

During the proof of Theorem 1.1, we will also obtain the following results. 

\begin{prop}\label{genus}
Under the assumptions in Theorem 1.1, set $N=|g|$ and let $\langle \alpha,\alpha\rangle=\frac{2K_0}{N_0}$ with $(K_0,N_0)=1$. Then $R=N/N_0$ is an integer. 
Suppose that $\tilde{V}(g)\cong V_{\Lambda}$ and $(R,N_0)=1$ or $(R,K_0)=1$. Then there is an $s\in \bQ$ such that an orbifold construction $\tilde{V}(g_s)$ from $V$ by using 
$g_s=\exp(2\pi is\alpha(0))$ satisfies the following conditions: \\
(1) $\rank(V_1)=\rank(\tilde{V}(g_s)_1)$; \\
(2) Coxeter numbers of simple Lie subalgebras of $\tilde{V}(g_s)_1$ are less than or equal to $R$; and \\
(3) $\widetilde{\tilde{V}(g_s)}(g_s)\cong V$. \\
In particular, $\tilde{V}(g_s)\otimes V_{\Pi_{1,1}}\cong V\otimes V_{\Pi_{1,1}}$, which 
means that $\tilde{V}(g_s)$ and $V$ belong to the same genus as defined in 
\cite{Mw}. 
\end{prop}

\begin{prop}[see Proposition \ref{ToLeech}] 
Let $V$ be a strongly regular holomorphic VOA of central charge $24$ with $V_1\not=0$. 
Then $V$ is realizable over $\bQ$.
\end{prop}

Our method is a generalization of Lorentzian constructions of the Leech lattice from Niemeier lattices, that is, we will use the Lorentzian lattice VA $V_{\Pi_{1,1}}$ 
of central charge $2$ and an isotropic VA.   
We also analyze the weight one subspaces of the irreducible twisted modules using the $\Delta$-operator defined by Li \cite{Li}. 
Except for several results in \cite{DM1,DM2},  the strange formula of  Freudenthal and De Vries (see Propositions \ref{prop:2.2}, \ref{prop:2.3} and Theorem \ref{thm:2.4}) and some knowledge of the Conway group $Co_0$, our method should be quite elementary. 
The key observation is that if $g=\exp(2\pi i \alpha(0))$ for some semisimple element $\alpha\in V_1$, then the structure of the irreducible $g$-twisted module $V^{(-\alpha)}$ defined by Li's $\Delta$-operator (see Proposition \ref{Delta_twist} and \S \ref{Sec:3.1}) is essentially determined by the action of $Y(\alpha,z)$ on $V$. By analyzing the actions of the subVOA $M(\alpha)$ generated by $\alpha$ and its commutant subVOA, we can obtain a lot of useful information about the structure of the twisted module $V^{(-\alpha)}$ and its character from the Lorentzian construction.
In \S 3, we first review some facts about the VOA $\tilde{V}(g)$ obtained by an orbifold construction from $V$ associated with an inner automorphism 
$g=\exp(2\pi i\alpha(0))$. We also describe a commutant VA of an isotropic VA 
$M((\alpha,1,\frac{\langle \alpha,\alpha\rangle}{2}))$ in a tensor product vertex algebra $V\otimes V_{\Pi_{1,1}}$ and consider 
its factor VA $V^{[\alpha]}$ divided by some ideal given by $\rho\to 0$. We call it a Lorentzian construction.   
In addition, we  study some relations between $\tilde{V}(g)$ 
and $V^{[\alpha]}$ and then identify them. 
This makes it easy to trace the change from $V$ to $\tilde{V}(g)$. 
One application of the Lorentzian construction is that if $\langle \alpha,\alpha\rangle=\frac{2K}{N}$ with 
$K,N\in \bZ$ and $N=|g|$, then we can also construct a holomorphic VOA $\bar{V}$ 
from $V$ by an automorphism $\exp(2\pi i\frac{N\alpha(0)}{K})$, which is isomorphic to $\tilde{V}(g)$. We also obtain a very precise relation between some lattice subVOAs in $V$ and $\tilde{V}(g)$ (cf. Theorem \ref{Lalpha}). 

In \S 4, we study the case when $\alpha$ is a $W$-element $\sum_{i=1}^t \frac{\rho_i}{h^{\vee}_i}$. 
As an application, we will show that if we repeat the orbifold constructions 
from $V$ by inner automorphisms using W-elements, then we can 
arrive at the Leech lattice VOA $V_{\Lambda}$ in finite steps. 
Therefore, we know that $V$ is realizable over $\bQ$.  We also give a proof for Proposition \ref{genus}. In \S 5, we study the irreducible $g$-twisted modules and $h$-twisted modules, where $g=\exp(2\pi i \alpha(0))$ and $h=\exp(2\pi i N\alpha(0)/K)$. We also prove 
one of our key theorems that the weight one subspace of 
$g$-twisted $V$-module $T^1$ is annihilated by $\alpha$. 
In \S 6, by setting $\langle\alpha,\alpha\rangle=\frac{2K_0}{N_0}$ with 
$(K_0,N_0)=1$ and $|g|=RN_0$, we will study the action of Conway group $Co_0$ and show that $\tilde{V}\cong V_{\Lambda}$
if $K_0$ or $N_0$ is $1$ or a prime number.  
In \S 7, we will prove Theorem \ref{Main} for the remaining cases.

\section{Preliminary results}

Throughout this paper, $(V,Y,\1,\omega)$ denotes a strongly regular holomorphic VOA (over $\bC$) of CFT-type with central charge $24$ and non-zero weight one space 
$V_1$.  By the definition of VOA, the weight one space $V_1$ has a Lie algebra structure defined by $0$-product $v(0)u$ and an invariant bilinear symmetric form 
$\langle \cdot,\cdot\rangle$ on $V_1$ given by $1$-product $u(1)v=\langle u,v\rangle\1$ for $v,u\in V_1$, where 
$Y(u,z)=\sum_{n\in \bZ} u(n) z^{-n-1}$ denotes a vertex operator of $u$ \cite{B}.  

For any positive integers $m,n$, we use $(m,n)$ to denote the g.c.d of $m$ and $n$.

\subsection{Structures of the weight one space $V_1$}

Our starting point is based on the following two results by 
Dong and Mason \cite{DM1,DM2}: 

\begin{prop}\label{prop:2.2}
Let $V$ be a strongly regular holomorphic VOA of CFT-type with central 
charge $24$ and $V_1\not=0$. Then,  \\
(1) $V_1$ is a semisimple Lie algebra or $\dim V_1=24$ and $V\cong V_{\Lambda}$, 
\\
(2) ${\rm rank}V_1\leq 24$ and $\dim V_1\geq 24$. 
In particular, $\rank(V_1)\geq 4$.\\
(3) If ${\rm rank}V_1=24$, then $V$ is isomorphic to a Niemeier lattice VOA. \\

\end{prop}

\begin{prop}\label{prop:2.3}
Assume that $V_1\cong \oplus_{j=1}^t \CG_{j,k_j}$ is semisimple, where  $\CG_{j,k_j}$ 
are simple Lie algebras  with the dual Coxeter number $h^{\vee}_j$ and level $k_j$ for $j=1, \dots, t$.  Then 
\begin{equation}
\frac{h_j^{\vee}}{k_j}=\frac{\dim V_1-24}{24}. 
\end{equation}
\end{prop}

From now on, we always assume that $V_1=\oplus_{j=1}^t \CG_{j,k_j}$ is semisimple. We also 
fix a set of simple roots of $\CG_{j,k_j}$ and 
define a Weyl vector $\rho_j$ as a half sum of all positive roots 
of $\CG_{j,k_j}$. The following result is well-known.

\begin{thm}[Strange formula of Freudenthal and De Vries] \label{thm:2.4}
\quad Let $\CG$ be a finite simple Lie algebra  and $\rho$ its Weyl vector. Then
$$  2h^{\vee} \dim \CG=24(\rho|\rho), $$
where we take the normalized Killing form $(\beta|\beta)=2$ 
for a long root $\beta$. 
\end{thm}

If we express it by the inner products $\langle \cdot,\cdot\rangle$ given by $1$-product, 
we have $\langle \cdot,\cdot \rangle=k_j (\cdot|\cdot)$ on $\CG_{j,k_j}$ and 
\begin{equation}\label{eq:2.2}
 k_jh_j^{\vee}\dim \CG_{j,k_j}=12\langle \rho_j,\rho_j\rangle.
\end{equation}

\subsection{Standard lifts} 
Let $F$ be an even lattice and $V_F$ the lattice VOA associated with $F$. Let $h\in O(F)$ and 
let $\hat{F}=\{\pm e^a\mid a\in F\}$ be a central extension of $F$ by $\{\pm 1\}$, where $O(F)$ 
is the symmetry group of $F$ and the group structure of $\hat{F}$ is given by the relation $(\ast)$:  $e^ae^b=(-1)^{\langle a,b\rangle}e^{b}e^{a}$. 
Let $\hat{h}\in O(\hat{F})$ be a standard lift of $h$, that is,  
$\hat{h}(e^{a})=e^{a}$ for $a\in F^{<h>}=\{v\in F\mid h(v)=v\}$. 
We note that we can view $\hat{h}$ as an automorphism of a lattice VOA $V_F$ 
by the definition of lattice VOAs.  
Generally, $|\hat{h}|=|h|$ or $2|h|$. The precise order of $\hat{h}$ is given in \cite[Lemma 12.1]{Bo92} as follows: 
\begin{lmm}
If $|h|$ is odd, then $|\hat{h}|=|h|$. 
If $n=|h|$ is even, then $|h|=|\hat{h}|$ if and only if 
$\langle h^{(n/2)}(a), a\rangle
\in 2\bZ$ for all $a\in F$. 
\end{lmm} 

\begin{lmm}\label{Fixedpf} If $h$ acts on an even lattice $F$ fixed point freely, then $|\hat{h}|=|h|$. 
\end{lmm}

\pr
We may assume that $n=|h|$ is even. For any $a\in F$, since $\sum_{i=0}^{n-1}h^i(a)=0$, 
we have 
$0=\sum_{i=0}^{n-1} \langle h^i(a),a\rangle=\langle a,a\rangle+\langle h^{n/2}(a),a\rangle+\sum_{i=1}^{n/2-1}\langle h^{j}(a)+h^{n-j}(a),a\rangle$. 
Since $\langle a,a\rangle\in 2\bZ$ and $\langle h^{n-j}(a),a\rangle=\langle a,h^j(a)\rangle\in \bZ$, 
we have $\langle h^{n/2}(a),a\rangle\in 2\bZ$. By the above lemma, we have 
$|\hat{h}|=|h|$.
\prend

\section{Orbifold construction by an inner automorphism}\label{sec:3}
In this section, we review the orbifold construction associated with an inner automorphism using Li's $\Delta$-operator. We will also describe the orbifold construction by an inner automorphism 
using the language of Lorentzian  VA  and commutant of isotropic VA.

\subsection{Haisheng Li's $\Delta$-operator}\label{Sec:3.1}
A good way to express twisted modules associated with 
an inner automorphism is by Haisheng Li's $\Delta$-operator \cite{Li}. 
Before we explain it, let's recall the orbifold construction 
from a holomorphic VOA $V$ by an automorphism $g$ of finite order $N$. 

Let $(T^m,Y^{T^m})$ be a $g^m$-twisted module of $V$. Since $V$ is holomorphic, 
$(T^m,Y^{T^m})$ is uniquely determined \cite{DLM2}.  In order to get a new VOA by an orbifold construction, it is necessary that  
$$ T^{1,0}:=T^1_{\bZ}=\{\sum w^j\in T^1 \mid \wt(w^j)\in \bZ\} \not=0. $$
We somehow need to show 
that $T^{1,0}$ is a simple current module as a $V^{<g>}$-module. 
In this case, $T^{m,0}:=(T^{1,0})^{\boxtimes m}\subseteq 
T^m$ is also a simple current module of $V^{<g>}$. 
Then we have to show $T^{m,0}\subseteq T^m_{\bZ}$ and that 
$$\widetilde{V}(g)=\oplus_{m=0}^{N-1} T^{m,0}$$ 
has a canonical VOA structure. We note that if $V$ is strongly regular, then 
all above conditions hold \cite{CMi,EMS}. 

In general, the explicit construction of twisted modules is not well-understood, 
but for an inner automorphism $g=\exp(2\pi i\alpha(0))$ 
with a semisimple element $\alpha\in V_1$, a $g$-twisted module $T^1$ can be explicitly constructed  
by Haisheng Li's $\Delta$-operator:   
\begin{equation}\label{eq:3.1}
\Delta(\alpha,z)=z^{\alpha(0)}\exp(\sum_{j=1}^{\infty}
\frac{\alpha(j)}{-j}(-z)^{-j} ). 
\end{equation}

First we recall the following result from \cite{Li} (see also \cite{DLM}).

\begin{prop}[{\cite[Proposition 5.4]{Li}}]\label{Delta_twist}    
Let $V$ be a VOA. Let $\alpha \in V_1$ such that the automorphism 
$g=\exp(2\pi i \alpha(0))$ has finite order on $V$. 
Let $(M, Y_M)$ be a $V$-module and
define $(M^{(\alpha)}, Y_{M^{(\alpha)}}(\cdot, z)) $ as follows:
\[
\begin{split}
& M^{(\alpha)} =M \quad \text{ as a vector space;}\\
& Y_{M^{(\alpha)}} (a, z) = Y_M(\Delta(\alpha, z)a, z)\quad \text{ for any } a\in V.
\end{split}
\]
Then $(M^{(-\alpha)}, Y_{M^{(-\alpha)}}(\cdot, z))$ is a
$g$-twisted $V$-module.
Furthermore, if $M$ is irreducible, then so is $M^{(-\alpha)}$.
\end{prop}

Namely, we can take $V$ as a base vector space of $T^1$ and define 
the vertex operator $Y^{T^1}(v,z)$ of $v\in V$ on $V=T^1$ by 
\[    
Y^{T^1}(v,z):=Y(\Delta(-\alpha,z)v,z).
\]
Then $(V, Y^{T^1})$ has a $g$-twisted module structure. Since $V$ is holomorphic, 
a  $g$-twisted module is uniquely determined and so $(V,Y^{T^1})$ is the desired 
$g$-twisted module.  
Similarly, we can define a $g^m$-twisted module $T^m=(V,Y^{T^m})$ by 
\[
Y^{T^m}(v,z):=Y(\Delta(-m\alpha,z)v,z)
\]
on $V$. The important fact is that the structures of these twisted modules are basically determined by the actions of $Y(\alpha,z)=\sum_{j\in \bZ} \alpha(j)z^{-j-1}$.

In order to see their structures, we separate the actions of $Y(\alpha,z)$ and the actions of the subVA which commute with 
the actions of $Y(\alpha,z)$.

Throughout this paper, $M(\alpha)$ denotes a subVA generated by 
$\alpha\in V_1$. 
It is known that $M(\alpha)$ is isomorphic to a Heisenberg VA of rank one given by a Lie algebra relation 
$[\alpha(n), \alpha(m)]=n\delta_{m+n,0}\langle \alpha,\alpha\rangle$.  
For a subspace $\CH$ of a Cartan subalgebra of $V_1$, we similarly define $M(\CH)$ 
as a subVA of $V$ generated by $\CH$, which is a Heisenberg 
VA of rank $\dim(\CH)$. 
We often use $\CH$ to denote a Cartan subalgebra of $V_1$. 

\begin{ntn} 
Let $V$ be a VA and $W$ a subVA of $V$. The commutant of $W$ in $V$ is given  by  
\begin{equation}\label{eq:3.2}  
{\rm Comm}(W,V)=\{v\in V\mid u(m)v=0 \quad \text{ for all } u\in W, m\geq 0\}. 
\end{equation} 
Throughout this paper, we use the following notation. 
\begin{equation}\label{eq:3.3}
X={\rm Comm}(M(\alpha),V) \quad \text{ and } \quad U={\rm Comm}(M(\CH),V).
\end{equation}
\end{ntn}

One advantage of using a rank one Heisenberg algebra 
is that the commutant VOA $X$ coincides with $\{v\in V\mid  \alpha(n)v=0, n\geq 0\}$, which can be   checked easily. 

From now on, we assume that $g=\exp(2\pi i \alpha(0))$ has a finite order $N$. 
In this case, the eigenvalues of $\alpha(0)$ on $V$ are in $\frac{\bZ}{N}$.  
Let $I$ be the ideal of $\bZ$ generated by the eigenvalues of $N\alpha(0)$ on $V$. Then $I=S\bZ$ for some positive integer $S$. Clearly $(S,N)=1$.  We set $K= N\langle \alpha, \alpha\rangle /2\in \bC$.  

\begin{prop}\label{Prop:6}
Let $V$ be a strongly regular holomorphic VOA. Then    
${\rm Comm}(X,V)$ is isomorphic to a lattice VOA $V_{\bZ \frac{N\alpha}{S}}$. Moreover, $K\in \bQ$. Set $K=K_1/K_2$ with $K_1,K_2\in \bN$ such that $(K_1, K_2)=1$. Then $K_2|N $ and $S^2|K_1$.
Similarly, there is an even lattice $L$ such that 
${\rm Comm}(U,V)\cong V_L$. 
\end{prop}

\pr
It is clear that the commutant ${\rm Comm}(X,V)$ is a VOA of central 
charge one and contains $M(\alpha)$ properly. So, it is a lattice VOA of rank one, 
say $V_{\bZ t\alpha}$. Since $\bZ t\alpha$ is an even lattice, $\langle t\alpha,t\alpha\rangle\in 2\bZ$. Also since $|g|=N$, $\langle t\alpha,N\alpha\rangle\in \bZ$ and so $t\in \bQ$, say $t=\frac{p}{q}$ with $p,q\in \bN$ such that 
$(p,q)=1$. Then the dual lattice of $\bZ \frac{p}{q} \alpha$ is 
$\bZ \frac{qN}{2pK} \alpha$ and $\langle \alpha, \frac{qN}{2pK} \alpha\rangle  = \frac{q}{p}$. Since $V$ is holomorphic, all values in $\frac{q}{p} \bZ $ will appear 
as eigenvalues of $\alpha(0)$ in $V$. Therefore, $\frac{q}{p} =\frac{S}{N}$ by 
the definition of $S$ and ${\rm Comm}(X,V)\cong V_{\bZ \frac{N\alpha}{S}}$.
Furthermore, $V_{\bZ \frac{N\alpha}{S}}$ is a
VOA and we have $\langle \frac{N\alpha}{S}, \frac{N\alpha}{S}\rangle = \frac{2KN^2}{NS^2}\in 2\bZ$ and thus, $K\frac{N}{S^2}\in \bN$. In particular, $K\in \bQ$, say $K = \frac{K_1}{
K_2}$ with $K_1,K_2\in \bN$ such that $(K_1,K_2) = 1$. Then we have $S^2K_2|K_1N$ and $S^2|K_1$ and $K_2|N$ as desired since $(S,N) = 1$ and $(K_2,K_1)=1$. 
Similarly, since ${\rm Comm}(U,V)$ contains $M(\CH)$ and it is $C_2$-cofinite,  it is also a lattice VOA.  
\prend

\begin{remark}\label{rem:3.3}
Set $\beta= \frac{\alpha}{S}$ and $h=\exp(2\pi i \beta(0))$. Then $g= h^S$. By the proof of Proposition \ref{Prop:6}, we know that the eigenvalues of $\beta(0)$ on $V$ are also in $\frac{\bZ}N$.  
Therefore, the order of $h$ also divides $N$ and thus is $N$. 
Recall that $(S,N)=1$; there are integers $k, \ell$ such that $kS+\ell N=1$. Then 
we have  $h= h^{kS+\ell N} =h^{kS}=g^k$ and $<g>=<h>$.  

Since we are mainly interested in the orbifold construction associated with $g$ and the resulting VOA depends only on the cyclic group $<g>$, by replacing $\alpha$ by $\beta=\alpha/S$ if necessary, we may assume that $S=1$ and ${\rm Comm}(X,V)\cong V_{\bZ N\alpha}$.   
\end{remark}

Although $X$-modules and $V_{\bZ N\alpha}$-modules are enough to consider 
the orbifold construction,  we will also 
explain the construction using $U$-modules and $V_{L}$-modules for the future notation. 

\begin{lmm}
$X$ is a strongly regular VOA. 
\end{lmm}

\pr
Since all simple modules of $V_{\bZ N\alpha}$ are simple current, 
there is a cyclic automorphism group $A$ of $V$ with finite order such that 
the fixed point subVOA $V^A$ coincides with $X\otimes V_{\bZ N\alpha}$. 
Then by \cite{CMi} and \cite{M}, $V^A$ is strongly regular and so is $X$. 
\prend
\medskip

We will consider an alternative construction of $\tilde{V}(g)$ (or one with the same character) using the Lorentzian lattice vertex algebra $V_{\Pi_{1,1}}$. 
In order to compare their structures, we will decompose $V$ as a direct sum of $M(\alpha)$-modules. By Remark \ref{rem:3.3}, we may assume that ${\rm Comm}(X,V)\cong V_{\bZ N\alpha}$. Then the dual lattice $(\bZ N\alpha)^*$ of $\bZ N\alpha$ is $(1/2K) \bZ\alpha$ and we have 
\begin{equation}\label{eq:3.4}
V=\bigoplus_{s\in \bZ} M(\alpha)\otimes e^{\frac{s}{2K}\alpha}\otimes X(\frac{s\alpha}{2K}),  
\end{equation}
where $M(\alpha)\otimes e^{\frac{s}{2K}\alpha} \otimes X(\frac{s\alpha}{2K})$ is an eigenspace of $\alpha(0)$ with eigenvalue $\frac{s}{N}$ and 
$X(\frac{s\alpha}{2K})$ denotes the multiplicity of simple $M(\alpha)$-module $M(\alpha)\otimes e^{\frac{s}{2K}\alpha}$. 

\begin{remark}\label{Xs}
We may view $X(\frac{s\alpha}{2K})$ as an $X$-module and $X=X(0)$; note that  
$ X(\frac{s\alpha}{2K})\cong X(\frac{s\alpha}{2K}+mN\alpha)$ 
as $X$-modules for any $m\in \bZ$.  For simplicity, we denote 
$X(\frac{s\alpha}{2K})$ by $X^{(s)}$. Then $X^{(s)} \cong X^{(s+2mKN)}$ for any $m\in \bZ$. 
Since ${\rm Comm}(X,V)\cong V_{\bZ N\alpha}$, we also have the decomposition 
\begin{equation}\label{VNadec}
V=\bigoplus_{s=0}^{2KN-1} V_{\frac{s}{2K}\alpha+ \bZ N\alpha}\otimes X^{(s)}.
\end{equation}
Since $V$ is holomorphic,  $X^{(s)}\neq 0$ for any $0\leq s\leq 2KN-1$ \cite{KM}.
\end{remark}

We next study the structure of the $g^m$-twisted module $T^m$ by using Li's $\Delta$-operator 
and compute the corresponding degree operator $L^{T^m}(0)$ on $T^m$. 

Since $\alpha(0)\omega=0$, 
$\alpha(1)\omega=\alpha$, $\alpha(2)\omega=0$, $\alpha(3)\omega=0$, we have 
$$ 
\Delta(-m\alpha,z)\omega=z^0(\omega-m\alpha z^{-1}
+\frac{m^2}{2}\langle \alpha,\alpha\rangle z^{-2})
$$
and so the degree operator $L^{T^m}(0)$ on $M(\alpha)\otimes e^{\frac{s\alpha}{2K}}\otimes X^{(s)}$ is given by  
\begin{equation}\label{LTm}
L(0)-\frac{m(s-mK)}{N}.
 \end{equation}
In particular, for $m=1$, we have 
$$L^{T^1}(0):=L(0)-\frac{s-K}{N}.$$ 

\begin{prop}
The $g$-twisted module $T^1$ has a non-zero integer weight element  
if and only if $K=\frac{N\langle \alpha,\alpha\rangle}{2}$ is an integer. 
\end{prop}

\pr 
If there is an integer weight element in $T^1$, then there is $s\in \bZ$ such 
that $\frac{s-K}{N}\in \bZ$ since $L(0)$ takes only integer values on $V$. 
In particular, $s-K\in \bZ$ and so $K\in \bZ$.  On the other hand, if 
$K$ is an integer, then for $s=K$, $L^{T^1}(0)$ has an integer eigenvalue. 
\prend

As we have seen, $K\in \bZ$ is a sufficient condition that we can execute an orbifold construction.  

\begin{remark} \label{lattice_twisted}
Note that $g= \exp(2\pi i \alpha(0))$ also acts on the lattice VOA $V_{\bZ N\alpha}$.  It is proved in \cite[Proposition 2.15]{DLM2} that for any irreducible $V_{\bZ N\alpha}$-module $V_{\frac{s}{2K}\alpha +\bZ N\alpha}$, the irreducible $g^m$-twisted module $(V_{\frac{s}{2K}\alpha+ \bZ N\alpha})^{(-m\alpha)}$ defined by $Y(\Delta(-m\alpha, z) \cdot, z)$ is isomorphic to $V_{\frac{s}{2K}\alpha -m\alpha+ \bZ N\alpha}$ . Therefore,  by \eqref{VNadec},  the $g^m$-twisted module $T^m$  can be decomposed as
\begin{equation}\label{twistmodule}
T^m= \bigoplus_{s=0}^{2KN-1} V_{\frac{s}{2K}\alpha -m\alpha + \bZ N\alpha}\otimes X^{(s)}
=\bigoplus_{s\in \bZ}M(\alpha)\otimes e^{\frac{s\alpha}{2K}-m\alpha}
\otimes X(\frac{s\alpha}{2K}).
\end{equation}
\end{remark}

\subsection{A construction using Lorentzian lattice VA}\label{Sec:Lorentzian}
In this subsection, we will discuss a factor VOA of a commutant VA using a Lorentzian lattice VA $V_{\Pi_{1,1}}$ and an isotropic VA, where $\Pi_{1,1}= \{(a,b)\mid a,b\in \bZ\}$ is a rank $2$ lattice with a Gram matrix $\begin{pmatrix} 0&-1 \cr -1&0\end{pmatrix}$. First we will recall a construction of the Leech lattice 
by using a Niemeier lattice and the Lorentzian lattice $\Pi_{1,1}$.  

\subsubsection{Niemeier lattice and Leech lattice}\label{SS:3.2.1}
Let $E$ be a Niemeier lattice and $E\not=\Lambda$.  
As it is well-known \cite[Chapter 26]{ConSl}, $E\oplus \Pi_{1,1}\cong \Lambda\oplus \Pi_{1,1}$.  
Let $\Gamma= \Gamma_1\perp \Gamma_2\perp \cdots \perp \Gamma_k$ be the decomposition of the root sublattice $\Gamma$ into irreducible components.  
Note that the irreducible components are of the type $A$, $D$ or $E$ and all the irreducible components have the same (dual)  Coxeter number $h$. 

Choose a fundamental system of roots for $\Gamma_i$ and let $\rho_i$ be the corresponding Weyl vector.  Set $\bar{\rho}= \sum \rho_i$. Then $\langle \bar{\rho}, \bar{\rho}\rangle =2h(h+1)$ and $\rho=(\bar{\rho},h, h+1) \in E\oplus \Pi_{1,1}$ 
is an isotropic element. In this case, $${\rho}^{\perp}/\bZ {\rho}\cong \Lambda,$$ 
where $ {\rho}^{\perp} = \{ (a, m,n) \in E\oplus \Pi_{1,1}\mid \langle 
(a, m,n), (\bar{\rho},h, h+1)\rangle=0\}$. 
\medskip

At the end of this section, we will consider the Niemeier lattice VOA $V_E$ and then 
translate the above arguments into the tensor product VA $V_E\otimes V_{\Pi_{1,1}}\cong V_{E\oplus \Pi_{1,1}}$. 
Then the above result can be written as   
\[
V_\Lambda \cong \mathrm{Comm}( M(\rho), V_E\otimes V_{\Pi_{1,1}})/ P, 
\]
where $P$ is a proper ideal of $\mathrm{Comm}( M(\rho), V_E\otimes V_{\Pi_{1,1}})$, which is defined as the kernel of 
the map  $\mathcal{R}: \mathrm{Comm}( M(\rho), V_E\otimes V_{\Pi_{1,1}}) \to \mathrm{Comm}( M(\rho), V_E\otimes V_{\Pi_{1,1}})$ such that $u \mapsto \lim_{\rho\to 0} u$ (cf. Remark \ref{rem:3.3.1}).

\subsubsection{Lorentzian lattice and the Lorentzian construction}\label{SS:3.2.2}
We discuss a generalization of the above construction to VOAs. 
Our setting is as follows: 

\medskip

Let $V$ be a VOA of CFT-type with $V_1\not=0$ and 
$\CH$ a subspace of a Cartan subalgebra of $V_1$ and $\alpha\in \CH$.    
Let $\Pi_{1,1}=\bZ r +\bZ q$ with $\langle r,r\rangle=\langle q,q\rangle=0$ and $\langle r,q\rangle= -1$.  Then the Lorentzian lattice VA $V_{\Pi_{1,1}}$ is given
by 
\[
V_{\Pi_{1,1}} =\bigoplus_{(m,n)\in \Pi_{1,1}}M(r, q)\otimes e^{(m,n)}
\]
with the formal elements $\{e^{(m,n)} \mid m,n\in \bZ\}$ and 
\[
V\otimes V_{\Pi_{1,1}} =\bigoplus_{m,n\in \bZ}\bigoplus_{s\in \bZ} M(\alpha, r,q)\otimes e^{\frac{s}{2K}\alpha}\otimes X^{(s)}\otimes e^{(m,n)}. 
\]
To simplify the notation, we denote $e^{a\alpha}\otimes e^{(m,n)}$ by $e^{(a\alpha,m,n)}$.

Set 
$\rho=(\alpha,1, \frac{\langle \alpha,\alpha\rangle}{2})
\in \CH\oplus \bC(\Pi_{1,1}) \subseteq (V\otimes V_{\Pi_{1,1}})_1$, 
which is an isotropic element and 
set $J={\rm Comm}( M(\rho), V\otimes V_{\Pi_{1,1}})$. 
Since $\Comm(J,V\otimes V_{\Pi_{1,1}})$ contains $M(\rho)$ and has central 
charge one, $\Comm(J,V\otimes V_{\Pi_{1,1}})\cong V_{\bZ N\rho}$ for 
$N\in \bC$, where $V_{\bZ 0\rho}$ denotes $M(\rho)$ if $N=0$.  
We are interested in the case $N\not=0$. In this case, 
since $e^{N\rho}=e^{(N\alpha,N,N\frac{\langle \alpha,\alpha\rangle}{2})}
\in V\otimes V_{\Pi_{1,1}}$, we have $N\in \bZ$ and $\langle \alpha,\alpha\rangle=\frac{2K}{N}$ for some 
$K\in \bZ$. Furthermore, since $X\subseteq J$, $e^{N\alpha}\in \Comm(X,V)$, 
where $X={\rm Comm}(M(\alpha),V)$. 

Set $P=\bC[\rho(-j):j>0]J+(\1-e^{N\rho})_{-1}J$. Note also that $e^{0\rho}=\1$.

The main aim of this section is to prove the following theorem. 

\begin{thm} \label{cosetVOA}
$J/P$ is a vertex operator algebra. We denote it by $V^{[\alpha]}$.
If $V$ is a strongly regular holomorphic VOA and we choose $\alpha$ 
as one in Remark \ref{rem:3.3}, then $V^{[\alpha]}\cong \tilde{V}(g)$ 
as an $X\otimes M(\alpha)$-module, where $g=\exp(2\pi i\alpha(0))$. 
\end{thm}

By a direct calculation, 
\[
\rho^\perp =\{ (a\alpha, m,n)\mid \langle \rho, (a\alpha, m,n)\rangle =0\} =\bC\rho+\bC\widehat{\alpha},
\]
where $\widehat{\alpha}= (\alpha,0, \frac{2K}N)$.  Since  
$\langle (\frac{(mK+nN)\alpha}{2K},m,n),(\alpha,1,\frac{K}{N})\rangle=0$, it is easy to check that 
$$\Ker( \rho(0))=\bigoplus_{m,n\in \bZ} M(\rho, \widehat{\alpha})\otimes X^{(mK+nN)}\otimes 
e^{(\frac{(mK+nN)\alpha}{2K},m,n)}.$$

Furthermore, for $n\geq 1$, we have 
$\rho(n)X^{(s)}e^{\bZ r+\bZ q}=0$.  
Hence, we have: 

\begin{lmm}\label{CommLoren}
$\displaystyle J={\rm Comm}(M(\rho), V\otimes V_{\Pi_{1,1}})
=\bigoplus_{m,n\in \bZ} M(\rho, \widehat{\alpha})\otimes 
X^{(mK+nN)} \otimes e^{(\frac{(mK+nN)\alpha}{2K},m,n)}.
$ 
\end{lmm}

We note that the eigenvalues of $N\widehat{\alpha}(0)$ on $J$ are 
in $\bZ R$, where $R=(K,N)$. Then 
$$\Comm(J,V\otimes V_{\Pi_{1,1}})=V_{\bZ N \rho}=\bigoplus_{n\in \bZ}M(\rho)
\otimes e^{(nN\alpha, nN, nK)}$$
is contained in $J$.

\begin{lmm}\label{ZJ}
For any $v\in J$, we have   
$e^{N\rho}_mv=0$ for all $m\geq 0$ and  
$e^{N\rho}_mv\subseteq \sum_{j=1}^{\infty} \rho(-j)J$ for $m\leq -2$. 
In particular, $V_{\bZ N\rho}$ is in the center of $J$. 
\end{lmm}

\pr 
We first note that $\wt(e^{N\rho}_mv)=\wt(v)-m-1$, since 
$\wt(e^{N\rho})=0$. 
We may assume $v=\xi\otimes e^{(s\alpha,a,b)}\in J$ with $\xi\in X(s\alpha)\otimes M(\alpha,\Pi_{1,1})$, 
because they span $J$. 
By the definition of vertex operators,
\[
\begin{split}
e^{N\rho}_mv & \in \bC[\rho(-j):j\in \bN]\bC[\rho(j):j\in \bN](\xi\otimes e^{(s\alpha,a,b)+N\rho})\\
& \in \bC[\rho(-j):j\in \bN](\xi\otimes e^{(s\alpha,a,b)+N\rho}).
\end{split}
\]
Since $\rho(0)v=0$, we have $\langle (s\alpha,m,n),\rho\rangle=0$. Furthermore, we have  $\wt(\xi\otimes e^{(s\alpha,m,n)+N\rho})=\wt(v)$ since 
$\wt(e^{N\rho})=0$.   
If $\wt(e^{N\rho}_mv)<\wt(v)$, the only possibility is $e^{N\rho}_mv=0$. If $\wt(e^{N\rho}_mv)>\wt(v)$, i.e. $m\leq -2$, then $e^{N\rho}_mv\in \sum_{j=1}^{\infty} \rho(-j)J$. 
\prend

\medskip

\begin{lmm}\label{idealP}
Set $P=\bC[\rho(-j):j>0]J+(\1-e^{N\rho})_{-1}J$. Then $P$ is a proper ideal of $J$. 
\end{lmm}
 
\pr
Since $\rho(-j)$ and $(\1-e^{N\rho})_{-1}$ commute with  the other operators in $J$, we have 
$u_mP\subseteq P$ for any $u\in J$ and $m\in \bZ$. For $w\in J$, 
\[
\begin{split}
(\rho(-j)u)_mw=&\sum_{i=0}^{\infty}\binom{-j}{i}\{ \rho(-j-i)u_{m+i}w-(-1)^{-j}u_{m-j-i}\rho(i)w \} \\ 
=&\sum\binom{-j}{i}\rho(-j-i)u_{m+i}w \in P   \quad \text{ and } \\
((\1-e^{N\rho})_{-1}u)_m w=&(\1-e^{N\rho})_{-1}(u_mw)+
\sum_{i=1}^{\infty}\binom{-1}{i}\{ e^{N\rho}_{(-1-i)}u_{m+i}w+u_{m-j-i}e^{N\rho}_iw \} \\
=&(\1-e^{N\rho})_{-1}(u_mw)+\sum_{i=1}^{\infty}\binom{-1}{i} e^{N\rho}_{(-1-i)}(u_{m+i}w) \in P. 
\end{split}
\]
It is easy to see that $\mathbf{1}\notin P$ and hence it is proper.
\prend 

\begin{remark}\label{rem:3.3.1}
 Since $\lim_{\rho\to 0}e^{N\rho}=\1$, we can view $P$ as the kernel of 
 the map  $\mathcal{R}: J \to J$ such that $u \mapsto \lim_{\rho\to 0} u$. 
\end{remark}

\begin{defn}\label{Valpha}
Define $V^{[\alpha]}=J/P$. 
We call $V^{[\alpha]}$ a VOA constructed by $\alpha$.
\end{defn}

\begin{lmm}\label{Dual}
Under the above setting, if 
$\langle\alpha,\alpha\rangle=\frac{2K}{N}$, then 
$V^{[N\alpha/K]}\cong V^{[\alpha]}$. 
\end{lmm}

\pr 
Clearly, an isotropic element 
$\frac{N}{K}\rho=(\frac{N\alpha}{K}, \frac{N}{K}, 1)$ defines the same commutant 
VOA $J$ and an ideal $P$ of $J$. On the other hand, the switching   $ r\leftrightarrow q$ of $\Pi_{1,1}=\bZ r+\bZ q$ induces an 
automorphism $\phi$ of $V\otimes V_{\Pi_{1,1}}$. By this automorphism, 
$\phi(\frac{N}{K}\rho)=(\frac{N\alpha}{K},1, \frac{N}{K})$ and so we have 
$V^{[\frac{N\alpha}{K}]}\cong V^{[\alpha]}$. 
\prend

We will next derive several useful expressions of $V^{[\alpha]}$. 
First we note that 
\[
(\frac{(mK+nN)\alpha}{2K},m,n)=m(\alpha,1,\frac{K}{N})+\frac{-mK+nN}{2K}(\alpha,0,\frac{2K}{N})=
m\rho+(\frac{-mK+nN}{2K})\widehat{\alpha}.
\]
Then we can decompose ${\rm Comm}(M(\rho), V\otimes V_{\Pi_{1,1}})$ into 
\begin{equation}\label{eq:3.8}
{\rm Comm}(M(\rho), V\otimes V_{\Pi_{1,1}})\\ \cong  
\bigoplus_{m=0}^{N-1}V_{m\rho +\bZ N \rho} \otimes \left (\bigoplus_{n\in \bZ}M(\widehat{\alpha})\otimes e^{(\frac{-mK+nN}{2K})\widehat{\alpha}}\otimes X^{(mK+nN)} \right ) 
\end{equation}

Since $\langle \rho, \widehat{\alpha}\rangle =0$ and 
$\rho$ is an isotropic element,  we have 
\[
\langle (\frac{(mK+nN)\alpha}{2K},m,n), (\frac{(mK+nN)\alpha}{2K},m,n)\rangle =\langle (\frac{-mK+nN}{2K})\widehat{\alpha}, (\frac{-mK+nN}{2K})\widehat{\alpha}\rangle.
\]
Therefore, as a graded vector space, we have the following result: 

\begin{thm}\label{tVd}  As a module of $V_{\bZ N \widehat{\alpha}} \otimes X$, 
\begin{equation}\label{eq:4.1}
\begin{split}
V^{[\alpha]} \cong &\bigoplus_{m=0}^{N-1}\bigoplus_{n\in \bZ}M(\widehat{\alpha})\otimes 
X^{(mK+nN)}\otimes e^{(-\frac{m\alpha}{2}+\frac{nN\alpha}{2K}, 0, n-\frac{mK}{N})}\\
\cong &\bigoplus_{m=0}^{N-1}\bigoplus_{n\in \bZ}M(\widehat{\alpha})\otimes e^{(\frac{-mK+nN}{2K})\widehat{\alpha}}\otimes X^{(mK+nN)}
\end{split} 
\end{equation}
where $\widehat{\alpha}=(\alpha,0,\frac{2K}{N})\in \bC \alpha\oplus \bC \Pi_{1,1}$
and $X^{(mK+nN)}=X(\frac{mK+nN}{2K}\alpha)$.  
\end{thm}

Since $X^{(2mNK+t)}=X^{(t)}$ and $\wt(e^{s\alpha})\geq 0$ for $m\in \bZ$ and $s\in \bQ$, the total set of conformal weights of 
$X^{(t)}\otimes e^{s\alpha}$ has a minimal value. 
Therefore, $V^{[\alpha]}$ is a vertex operator algebra. 

\begin{remark}\label{rem:3.10} 
Set $\frac{K}{N}=\frac{K_0}{N_0}$ with $(N_0,K_0)=1$. 
Then $-m'K+n'N=-mK+nN$ if and only if 
there is a $t\in \bZ$ such that $m'=m+tN_0$ and $n'=n+tK_0$. Therefore, the element $\frac{-mK+nN}{2K}\widehat{\alpha}$ does not determine the  
pair $(m,n)$ uniquely if $R=\frac{N}{N_0}\not=1$. As 
$X^{(mN_0K-mK_0N)}=X^{(0)}={\rm Comm}(M(\alpha),V)\subseteq V^{<g>}$, we also know that  $e^{\frac{(mN_0K+mK_0N)\alpha}{2K}}
=e^{mN_0\alpha}\in {\rm Comm}(X^{(0)},V^{[\alpha]})$ for all $m\in \bZ$. 
It is also easy to check the converse. Therefore, we have 
\[
{\rm Comm}(M(\alpha),V^{[\alpha]} )
=\bigoplus_{m=0}^{R-1} X^{(2mRN_0K_0)}\quad \text{ and } \quad 
{\rm Comm}(X^{(0)},V^{[\alpha]})\cong V_{\bZ N_0\widehat{\alpha}}.
\] 
\end{remark}

\begin{prop}\label{R=1}
If $R=1$, then $(V^{[\alpha]})^{[\alpha]}\cong V$. 
\end{prop}

\pr Suppose $R=(N,K)=1$. Then, for any $s\in \bZ$, there is the unique pair $(m,n)\in \bZ^2$ such that 
$0\leq m\leq N-1$ and $s=mK+nN$. Then by comparing (3.9) and (3.4), 
we have $(V^{[\alpha]})^{[\alpha]}=\oplus_{m=0}^{N-1}\oplus_{n\in \bZ}M(\widehat{\hat{\alpha}})
\otimes e^{\frac{mK+nN}{2K}\widehat{\hat{\alpha}} }\otimes X(\frac{mK+nN}{2K}\alpha)$. 
By identifying $\widehat{\hat{\alpha}}=\alpha$, we have the desired result. 
\prend

\subsection{Orbifold construction and Lorentzian construction}\label{sec:3.3}
We will use Li's $\Delta$-operator to compare the above VOA $V^{[\alpha]}$ with the 
VOA $\tilde{V}(g)$ obtained by the orbifold construction with an automorphism $g=\exp(2\pi i\alpha(0))$.   
Recall that the underlying vector space of the irreducible $g^m$-twisted module $T^m$ is also $V=\oplus_{s\in \bZ} M(\alpha)\otimes X^{(s)}\otimes e^{\frac{s\alpha}{2K}}$.
The degree operator $L^{T^m}(0)$  acts on $M(\alpha)\otimes e^{\frac{s\alpha}{2K}}\otimes X^{(s)}$ by  $L(0)-\frac{m(s-mK)}{N}$. 
Set  
\[
Z^m:=\oplus_{s\in \bZ} M(\hat{\alpha})\otimes X^{(s)}\otimes 
e^{\frac{s\alpha}{2K}}\otimes e^{(m,\frac{s-mK}{N})}
\subseteq V\otimes V_{\frac{1}{N}\Pi_{1,1,}}. 
\]
Then, the grading on $M(\alpha)\otimes X^{(s)}
\otimes e^{\frac{s\alpha}{2K}}\subseteq T^m$ coincides with 
the grading on $Z^m$. 

Moreover, $\Delta(-m\alpha,z)u=u$ for any $u\in X={\rm Comm}(M(\alpha),V)$ and hence 
$Y^{T^m}(u,z)=Y(u,z)$, which coincides with the action of $u$ on $Z^m$. Since $\Delta(-m\alpha,z)\alpha=\alpha-m\frac{2K}{N}\1 z^{-1}$, we have   
\begin{equation}\label{aTm}
Y^{T^m}(\alpha,z)=Y(\alpha,z)-m\frac{2K}{N}\1 z^{-1}.
\end{equation} 
In particular, there is no change on the operators $\alpha(n)$ of $\alpha$ 
except the grade preserving operator $\alpha(0)$ and the grade preserving operator $\alpha(0)$ on $T^m$ 
is given by  $\alpha(0)-m\frac{2K}{N}$ as an operator on $V$, which coincides with the action of 
$\widehat{\alpha}(0)$ on 
$M(\alpha)\otimes X^{(s)}\otimes e^{\frac{s\alpha}{2K}}\otimes e^{(m,\frac{s-mK}{N})}$. Note that 
$\langle \widehat{\alpha}, \widehat{\alpha}\rangle = \langle \alpha, \alpha \rangle$.  
By identifying $\widehat{\alpha}$ and $\alpha$, we have the following isomorphisms of $M(\alpha)\otimes X$-modules:
\begin{equation}\label{eq:3.5}
\begin{array}{l}
\phi_1: T^1 \to Z^1=\bigoplus_{s\in \bZ} M(\widehat{\alpha})\otimes e^{\frac{s\alpha}{2K}}\otimes X^{(s)}\otimes e^{(1,\frac{s-K}{N})}\subset 
V\otimes V_{\frac{1}{N}\Pi_{1,1}}, \cr
\phi_m: T^m \to Z^m=\bigoplus_{s\in \bZ} M(\widehat{\alpha})\otimes e^{\frac{s\alpha}{2K}}\otimes X^{(s)}\otimes e^{(m,\frac{s-mK}{N})} 
\subseteq V\otimes V_{\frac{1}{N}\Pi_{1,1}}. 
\end{array} 
\end{equation}
for $m\in \bZ$ by \eqref{LTm} and the above arguments.   

Using the above notation, we also have  
$$\phi_0: V \to Z^0:=\bigoplus_{s\in \bZ}M(\widehat{\alpha})\otimes e^{\frac{s\alpha}{2K}}
\otimes X^{(s)}\otimes e^{(0,\frac{s}{N})} \subseteq V\otimes V_{\frac{1}{N}\Pi_{1,1}}.$$  
In particular, the subspace $T^{1,0}= T^1_{\bZ}$ of $T^1$ with integer weights corresponds to the subspace of $Z^1$ for which $(s-K)/N\in \bZ$, i.e., $s=K+nN$ for some $n\in \bZ$ and we have 
\begin{equation}
\phi_1:T^{1,0}\cong Z^{1,0}:=\bigoplus_{n\in \bZ} M(\widehat{\alpha})\otimes e^{\frac{(K+nN)\alpha}{2K}}\otimes X^{(K+nN)}\otimes e^{(1,n)}
\end{equation}

Under our assumption, it is known that $T^{1,0}=T^1_{\bZ}$ is a simple current 
and $T^{m,0}=(T^{1,0})^{\boxtimes m}$ for $m\geq 0$, where $\boxtimes$ denotes a fusion product. 
On the other hand, if we set 
\begin{equation}\label{eq:3.6}
Z^{m,0}:= \bigoplus_{n\in \bZ} M(\widehat{\alpha})\otimes e^{\frac{(mK+nN)\alpha}{2K}}\otimes X^{(mK+nN)}\otimes e^{(m,n)}, 
\end{equation}
then it is easy to see the products of elements 
in $Z^{m_1,0}$ and $Z^{m_2,0}$ belongs to $Z^{m_1+m_2,0}$ and 
$Z^{m,0}=\phi_m(T^{m,0})$.  Moreover, $Z^{m,0}+P = Z^{m',0}+P$ if and only if $m\equiv m'\mod N$.

\begin{remark}\label{Vmodules}
We should note that not only as $M(\alpha)\otimes X$-modules, 
but $T^1$ and $Z^1$ are also isomorphic  as $V$-modules. 
However, in the proof of our main theorem, we only use  the characters of twisted modules and so we will not give a 
precise proof of $V$-isomorphisms, but the proof of 
$V^{<g>}$-isomorphisms. 
\end{remark}

An orbifold construction 
from $V$ by using $g$ is to give a VOA structure on $\oplus_{m=0}^{N-1}T^{m,0}$ preserving the $V^{<g>}$-module structures and the fusion products. 

Note that $\phi_0$ maps $\alpha\in \CH$ to $\widehat{\alpha}=(\alpha,0,\frac{2K}{N})\in (V\otimes V_{\Pi_{1,1}})_1$.

For an $M(\alpha)$-singular vector $e^{\frac{s\alpha}{2K}}\otimes v \in e^{\frac{s\alpha}{2K}}\otimes X^{(s)}$, 
$\Delta(-\alpha,z)e^{\frac{s\alpha}{2K}}\otimes v=z^{-\frac{s}{N}}e^{\frac{s\alpha}{2K}}\otimes v$ 
and 
$$Y^{T^m}(e^{\frac{s\alpha}{2K}}\otimes v ,z)=z^{-\frac{ms}{N}}Y(e^{\frac{s\alpha}{2K}}\otimes v,z). $$
Therefore, the actions of $e^{\frac{s\alpha}{2K}}\otimes v$ on $V$ and 
on $T^m$ are the same except for the grading shift of $-ms/N$.  
On the other hand, it is easy to see  
$$Z^0=\bigoplus_{s\in \bZ} M(\widehat{\alpha})\otimes e^{\frac{s\alpha}{2K}}\otimes X^{(s)}\otimes 
e^{(0, s/N)}=\bigoplus_{s\in \bZ} M(\widehat{\alpha})\otimes X^{(s)}
\otimes e^{\frac{s\widehat{\alpha}}{2K}}$$
is a VOA which is isomorphic to $V$ and 
\begin{equation}\label{eq:3.7}
V^{<g>}\cong Z^{0,0}=\bigoplus_{n\in \bZ}M(\widehat{\alpha})\otimes e^{\frac{nN\alpha}{2K}}\otimes X^{(nN)}\otimes e^{(0,n)} \subseteq V\otimes V_{\Pi_{1,1}}.
\end{equation} 
If we consider the action of $V^{<g>}$ on $T^{m,0}$, the above difference 
can be covered by the action of 
$e^{\frac{kN\widehat{\alpha}}{2K}}\otimes v (= e^{\frac{kN \alpha}{2K}}\otimes v \otimes e^{(0, k)})$ on $V\otimes e^{(m, n)}$. 
Since $\{M(\alpha), X, e^{\frac{s\alpha}{2K}}\otimes X^{(s)}\mid s\in N\bZ\}$ generate $V^{<g>}$, 
we know that $\phi_m$ are isomorphisms of $V^{<g>}$-modules. 
Since $T^{1,0}$ is a simple current, $T^1\cong V\boxtimes_{V^G}T^{1,0}$, 
on which the actions of $V$ is uniquely determined up to isomorphisms. In this case, since $Z^{1,0}\cong T^{1,0}$, $Z^{1,0}$ is also simple current and it is easy to see $Z^1=Z^0\boxtimes_{Z^{0,0}}Z^{1,0}$. 
Therefore, by viewing $V\cong Z^0$, we are able to identify $T^1$ with $Z^1$. 
This completes the proof of $T^1\cong Z^1$ as $V^{<g>}$-modules.

From now on, we will identify $T^{m,0}$ in the $g^m$-twisted module $T^m$ with 
$Z^{m,0} \subseteq  V\otimes V_{\Pi_{1,1}}$ as $M(\alpha)
\otimes X$-modules. Using this identification, we will see that there is a VOA structure on $\oplus_{m=0}^{N-1} T^{m,0}$ by realizing it as a factor vertex subalgebra as discussed in Theorem \ref{cosetVOA}. This gives an alternative way to construct a holomorphic VOA $\tilde{V}$ by the orbifold construction associated with $V$ and $g=\exp(2\pi i\alpha(0))$.   Recall that   
$$T^{m,0}\cong Z^{m,0}=\bigoplus_{n\in \bZ} M(\widehat{\alpha})\otimes X^{(mK+nN)}
\otimes e^{(\frac{(mK+nN)\alpha}{2K},m,n)}.$$

Let $\rho=(\alpha,1,K/N)$. By Lemma \ref{CommLoren}, we have 
$$
J={\rm Comm}(M(\rho), V\otimes V_{\Pi_{1,1}})
=\bigoplus_{m,n\in \bZ} M(\rho, \widehat{\alpha})\otimes 
X^{(mK+nN)} \otimes e^{(\frac{(mK+nN)\alpha}{2K},m,n)}, 
$$ 
and 
$$V^{[\alpha]}={\rm Comm}(M(\rho), V\otimes V_{\Pi_{1,1}})/ P$$ 
is a vertex algebra, where $P=\bC[\rho(-j):j>0]J+(\1-e^{N\rho})_{-1}J$. 

It is clear that $T^{m,0} \subseteq {\rm Comm}(M(\rho), V\otimes V_{\Pi_{1,1}})$ for any $m\in \bZ$.  Since $M(\widehat{\alpha}, \rho)\, (=M(\frac{N\alpha}{2K}+q, \frac{\alpha}{2}+r))$ is contained in the vertex subalgebra generated by  $e^{(\frac{(mK+nN)\alpha}{2K},m,n)}$ and $e^{(\frac{(-mK-nN)\alpha}{2K},-m,-n)}$ for $m,n\in \bZ$,    
${\rm Comm}(M(\rho), V\otimes V_{\Pi_{1,1}})$ coincides with the vertex subalgebra 
${\rm VA}(T^{1,0})$ generated by $T^{1,0}$. 

As a consequence, $V^{[\alpha]}={\rm Comm}(M(\rho), V\otimes V_{\Pi_{1,1}})/P$ defines a VOA structure on $\oplus_{m=0}^{N-1} T^{m,0}$, which is isomorphic to the VOA $\tilde{V}(g)$ 
obtained by the orbifold construction from $V$ and $g=\exp(2\pi i \alpha(0))$ as 
an $M(\alpha)\otimes X$-module. This proves Theorem \ref{cosetVOA}.  
In particular, if $V^{[\alpha]} \cong V_{\Lambda}$, then $\tilde{V}(g)\cong V_{\Lambda}$.

Now consider a Niemeier lattice VOA $V_E$. The weight one subspace $(V_E)_1$ is a direct sum $\oplus 
\CG_i$ of simply laced Lie algebras $\CG_i$ with the same (dual) Coxeter number $h$  and the levels are all one. 

By the above arguments and the discussion in \S \ref{SS:3.2.1}, we have the following result:

\begin{prop}\label{Niemeier}
An orbifold construction from $V_E$ by an automorphism $\exp(2\pi i\rho(0)/h)$ associated with a  
W-element $\rho/h$ of $(V_{E})_1$ gives the Leech lattice VOA $V_{\Lambda}$. 
\end{prop}

From now on, we always assume that the inner automorphism $g=\exp(2\pi i\alpha(0))$ has a finite order $N$ and $K=N\langle \alpha, \alpha\rangle/ 2$ is an integer. We also set  $R=(N,K)$, the g.c.d of $N$ and $K$. 

\subsection{Structure of ${\rm Comm}(M(\CH))$ and the reverse automorphism}
Let $\CH$ be a subspace of a Cartan subalgebra of $V_1$ and assume 
$\alpha\in \CH$ and set $U={\rm Comm}(M(\CH),V)$. 
Then by the same argument for $\bC \alpha$, 
${\rm Comm}(U,V)$ is isomorphic to a lattice VOA $V_L$ 
for an even lattice $L$. 
Set $\widehat{\CH}=\{(\beta, 0, \langle \beta,\alpha\rangle)\mid \beta\in \CH\}\subseteq Z^{0,0}$ and denote 
$\widehat{\beta}=(\beta, 0, \langle \beta,\alpha\rangle)\in \widehat{\CH}$ for $\beta\in \CH$. 
Note that we can identify $\widehat{\beta}$ and 
$\widehat{\CH}$ with $\beta$ and $\CH$, respectively. 
We also identify $U$ with 
$U\otimes e^{(0,0)}\subseteq {\rm Comm}(M(\widehat{\alpha}),V^{[\alpha]})$. 
One aim in this subsection is to show that if $\CH$ is a Cartan subalgebra of $V_1$, then a Cartan subalgebra of $V^{[\alpha]}_1$ 
containing $\CH$ is uniquely determined. 

For $u\in \bC\otimes_\bZ L=\bC \alpha+(\bC\alpha)^{\perp}=\CH$,     
we often denote $u=u'+\frac{<u,\alpha>N_0\alpha}{2K_0}$ 
with $u'\in (\bC\alpha)^{\perp}$. 
We will use a similar notation for 
$\widehat{\CH}= \bC \widehat{\alpha}+ (\bC\widehat{\alpha})^{\perp}$. 
Then we have:
$$V=\bigoplus_{u'+a\alpha\in L^{\ast}}U(u'+a\alpha)\otimes M(\CH)\otimes e^{u'+a\alpha} $$
as $U\otimes M(\CH)$-modules and 
$$V^{[\alpha]}=\bigoplus_{u'+\frac{(mK+nN)\alpha}{2K}\in L^{\ast}, m,n\in \bZ} 
U( u'+\frac{(mK+nN)\alpha}{2K})\otimes M(\widehat{\CH})\otimes e^{u'+\frac{(-mK+nN)\widehat{\alpha}}{2K}} .$$

\begin{remark}
Recall that ${\rm Comm}(U,V)\cong V_L$ is a lattice VOA. Thus, we can decompose $V$ as 
\[
V= \bigoplus_{u+L\in L^*/L} V_{u+L} \otimes U(u).
\]
Since $V$ is holomorphic, all irreducible modules of $V_L$ must appear as a submodule of $V$ \cite{KM}. That means $U(u)\neq 0$ for all $u\in L^*$. Note also that 
$U(u)\cong U(v)$ if $u-v\in L$. 
\end{remark}

Next we define 
$$ U^{[\alpha]}={\rm Comm}(M(\widehat{\CH}),V^{[\alpha]}).$$ 
Clearly, $U=U\otimes e^{(0,0)}\subseteq U^{[\alpha]}$. Moreover, 
${\rm Comm}(U^{[\alpha]},V^{[\alpha]})\cong V_{\tilde{L}}$
for some even lattice $\tilde{L}$.

Since $\langle \alpha, \alpha\rangle= \langle \widehat{\alpha} , \widehat{\alpha}\rangle$, 
$V_{\bZ N \alpha}\cong V_{\bZ N\widehat{\alpha}}$ as lattice VOAs. 
Under the identification $\CH=\widehat{\CH}$, we may view both $L$ and $\tilde{L}$ 
as subsets of $\CH$. 

\begin{ntn}
Set $\bar{L}=\{u\in L\mid \langle u,N_0\alpha\rangle\in \bZ\}$. 
For $u\in \bar{L}$, we
define $u^{[\alpha]} = u- m\alpha$ when $\langle u, N_0\alpha\rangle \equiv mK_0 \mod N_0$ and $0\leq m\leq N_0-1$. Set $\bar{L}^{[\alpha]}= \{ u^{[\alpha]}| u \in \bar{L}\}$. 
 Then $\bar{L}^{[\alpha]}$ is also an even lattice and $\det(\bar{L})=\det(\bar{L}^{[\alpha]})$. 
 Note that $N\alpha \in \bar{L}$ and $(N\alpha)^{[\alpha]} =N\alpha$.
\end{ntn}

\begin{thm}\label{Lalpha}
Suppose $N_0\alpha\in L^*$. Then $\bar{L}=L$ and we have 
\[
\Comm(M(\widehat{\CH}),V^{[\alpha]})=\oplus_{m=0}^{R-1} U(mN_0\alpha).
\]
Moreover, $\tilde{L} = L^{[\alpha]}+ \bZ N_0\alpha$.  
\end{thm}

\pr  
We first note that $U(mN\alpha)=U$ for $m\in \bZ$ and $mN_0\alpha=\frac{mN_0K+mK_0N}{2K}\alpha$.  Since  $N_0\alpha\in L^*$, the subspace $M(\CH)\otimes e^{mN_0\alpha}\otimes U(mN_0\alpha)\neq 0$ is contained in $V$. It transforms to  
\[
M(\widehat{\CH})\otimes e^{\frac{-mN_0K+mK_0N}{2K}\widehat{\alpha}}\otimes U(mN_0\alpha) = M(\widehat{\CH})\otimes e^{(0,0,0)}\otimes U(mN_0\alpha)
\] 
in the twisted part $T^{mN_0,0}$.  Therefore, $0\neq U(mN_0\alpha)\subseteq V^{[\alpha]}$. 
On the other hand, it is easy to see that a $U$-coset $U(\beta)$ 
for $\beta\in L^{\ast}$ is contained in $U^{[\alpha]}$ if and only if 
 $\beta-mN_0\alpha\in L$ for some $m\in \bZ$. 
Therefore, we have the desired result for $U^{[\alpha]}$. 

The second statement follows from the same arguments. Note that 
$  M(\CH)\otimes e^{\beta}\otimes U(0)\subseteq V$ if and only if
$M(\CH)\otimes e^{\beta -m\alpha +tN_0\alpha }\otimes U(0) \subseteq V^{[\alpha]}$ when   $\langle \beta, N_0\alpha\rangle\equiv m K_0\mod N_0$ and $t\in \bZ$.     
\prend

\medskip

\begin{defn}
We define  $\widetilde{g}\in \Aut(\widetilde{V}(g))$ such that $\tilde{g}$ acts on the 
$V^{<g>}$-submodule $T^{m,0}$ arisen from a $g^m$-twisted module $T^m$ of $V$ as a scalar multiple of $\xi^m$, where $\xi=\exp(2\pi i/N)$. We call $\tilde{g}$ \textbf{the reverse automorphism} of $g$. Note that 
 $\widetilde{g}$ acts on $M(\widehat{\alpha})\otimes e^{(\frac{-mK+nN}{2K})\widehat{\alpha}}\otimes X^{(mK+nN)}$ as a scalar multiple 
by $\xi^m$ in our notation.  
\end{defn}

As an important observation, we have:

\begin{lmm}\label{Lalpha1}
The reverse automorphism $\widetilde{g}$ acts on $M(\widehat{\alpha})\otimes X(mN_0\alpha)$ 
(and also on $U(mN_0\alpha)$) as a multiple by 
a scalar $\xi^{mN_0}$.
\end{lmm}

\pr
Recall that  $mN_0 = \frac{mN_0K+ mK_0N}{2K}$. Then by our convention, 
$\widetilde{g}$ acts on
\[
M(\widehat{\alpha})\otimes e^{\frac{-mN_0K+mK_0N}{2K}\widehat{\alpha}}\otimes X(mN_0\alpha) = M(\widehat{\alpha})\otimes e^{(0,0,0)}\otimes X(mN_0\alpha)
\] 
as a multiple of scalar  $\xi^{mN_0}$ as desired.  
\prend

\begin{remark}\label{Reverseq}
By \eqref{eq:3.6} and \eqref{eq:3.7}, it is easy to see that the reverse automorphism $\tilde{g}$ is induced from the automorphism 
$\exp(-2\pi i q(0)/N) \in \Aut( V\otimes V_{\Pi_{1,1}})$, where 
$\Pi_{1,1}=\bZ r +\bZ q$ with $\langle r,r\rangle=\langle q,q\rangle=0$ and $\langle r,q\rangle= -1$.  
\end{remark}

By identifying $V^{<g>}$ with $(V^{[\alpha]})^{<\tilde{g}>}$, we may view them as the same space. In this case, we may also identify $\widehat{\alpha}$ with $\alpha$ and $\widehat{\CH}$ with $\CH$. 
From now on, we use the same notation $\alpha$ and $\CH$ for 
$\hat{\alpha}$ and $\hat{\CH}$.

\begin{prop}\label{abelian}
If $\CH$ is a Cartan subalgebra of $V_1$, then 
$U^{[\alpha]}_1=\Comm(M(\CH),V^{[\alpha]})_1$ is an abelian Lie algebra. 
In particular, $U^{[\alpha]}_1\oplus \CH$ 
is the unique Cartan subalgebra of $V^{[\alpha]}_1$ containing $\CH$. 
\end{prop}

\pr 
Since $(V^{[\alpha]})^{<\tilde{g}>}_1=V^{<g>}_1$ and 
$\Comm(M(\CH),V_1^{<g>})=0$,  we have that $\tilde{g}$ acts on 
$\Comm(M(\CH),V^{[\alpha]})_1$ fixed point freely and so 
$\Comm(M(\CH),V^{[\alpha]})_1$ is solvable \cite[Theorem 2]{Zha}. 
Then $\Comm(M(\CH),V^{[\alpha]})_1$  is a solvable subalgebra of a semisimple Lie algebra and hence it is abelian.   
\prend
\medskip

From now on, we denote $\Comm(M(\CH),V^{[\alpha]})_1\oplus \CH$ and $\Comm(M(\CH),V^{[\alpha]})_1$ by $\CH^{[\alpha]}$ and $\CH^{\perp}$, respectively. 
As a corollary, we have the following: 

\begin{cry}\label{Tilde} Suppose $N_0\alpha\in L^*$. Then 
$R=1$, or $\CH^{[\alpha]}\supsetneq \CH$, or $\CH^{[\alpha]}=\CH$ and 
$U^{[\alpha]} \supsetneq U$.  
\end{cry}

\begin{lmm}\label{KeyLemma3}
If $V^{[\alpha]}\not\cong V_{\Lambda}$, then there is a positive root system of $V^{[\alpha]}_1$ 
on which $\widetilde{g}$ acts.
\end{lmm}

\pr Since $\Comm(\CH, V^{[\alpha]})_1=\CH^{\perp}$ does not contain any root vectors,    $\CH\not\subseteq <\beta>^{\perp}$ for any root $\beta$. Since 
the number of roots in $V^{[\alpha]}_1$ is finite, there is $\gamma\in \bQ L\subseteq \CH\subseteq (V^{[\alpha]})^{<\tilde{g}>}$ such that $\langle \beta,\gamma\rangle\not=0$ for all roots $\beta$ of $\CH^{[\alpha]}$. 
We call a root $\beta$ positive if $\langle \beta,\gamma\rangle>0$. 
Then $\widetilde{g}$ acts on the set of positive roots as a permutation. \prend

\begin{thm} \label{autocommute}
Let $V$ be a VOA of CFT-type and let $\CH$ be a Cartan subalgebra of $V_1$ and let $\alpha\in \CH$ with $\langle \alpha, \alpha\rangle\in \bQ$. 
Assume $V_L\cong \Comm(\Comm(M(\CH),V),V)$ and let 
$N$ be the least positive integer such that $N\alpha\in L$. Let $f\in \Aut(V)$ 
such that $f$ acts trivially on $V_{\bZ N\alpha}$. 
Then $\hat{f} =f\otimes \mathrm{id} \in \Aut( V\otimes V_{\Pi_{1,1}})$ 
induces an automorphism $\hat{f}$ on $V^{[\alpha]}$ and $[\hat{f}, \tilde{g}]=1$, 
where $\tilde{g}$ is the reverse automorphism of 
$g=\exp(2\pi i\alpha(0))$ on $V^{[\alpha]}$. 
\end{thm}

\pr 
First we recall from Definition \ref{Valpha} that the VOA 
$V^{[\alpha]}$ can be identified  with  
\[
\Comm(M(\rho_{\alpha}), V\otimes V_{\Pi_{1,1}})/P,  
\]
where   $\rho_{\alpha}= (\alpha, 1, \frac{\langle \alpha, \alpha \rangle}{2})$ and $P=\bC[\rho_\alpha(-j):j>0]J+(\1-e^{N\rho_{\alpha}})_{-1}J$ with $J= \Comm(M(\rho_{\alpha}), V\otimes V_{\Pi_{1,1}})$.

Set $\hat{f}= f \otimes \mathrm{id}$, which is an automorphism  of $V \otimes V_{\Pi_{1,1}}$. Since $f$ acts trivially on $V_{\bZ N\alpha}$, $\hat{f} $ fixes  $\rho_{\alpha}$ and $e^{N\rho_\alpha}$.  Therefore, 
$\hat{f}$ preserves the commutant  $\Comm(M(\rho_{\alpha}), V\otimes V_{\Pi_{1,1}})$ and $P$; thus it induces an automorphism $\hat{f}$ on $V^{[\alpha]}$. 
Furthermore, since $\tilde{g}\in \Aut(V^{[\alpha]})$ 
is the restriction of $\exp(-2\pi iq(0)/N)$ on $J$ as we explained in Remark \ref{Reverseq} and $\hat{f}$ fixes $q$, 
we have $[\hat{f},\tilde{g}]=1$ on $V^{[\alpha]}$ as desired. 
\prend

We next treat the case where $V^{[\alpha]}=V_E$ for a Niemeier lattice $E$. 
We also assume $\alpha\in L^{\ast}$ and $\langle\alpha,\alpha\rangle\in 2\bZ$. 
Namely, $N_0=1$ and thus $N=R$. In this case, since $\tilde{g}\in \Aut(V_E)$, there is 
$\tau\in O(E)$ and 
$\delta\in \bC E$ such that $\tilde{g}=\hat{\tau}\exp(\delta(0))$, where $\hat{\tau}$ 
is a standard lift of $\tau$.  
Recall that $R=N$ is the least positive integer such that 
$R\alpha\in L$.  
Then $R=|\exp(2\pi i\alpha(0))|=|\tilde{g}|$. Since $\alpha=N_0\alpha\in L^{\ast}$, 
we have $\oplus_{m=0}^{R-1} U(m\alpha)=\Comm(M(\CH),V^{[\alpha]})=V_{E_{\tau}}$
and $\tilde{L}=L+\bZ \alpha=E^{<\tau>}$, where 
$E_{\tau}=E\cap (E^{<\tau>})^{\perp}$. 
We note $\CH^{\perp}=\bC E_{\tau}$. Since $\tau$ acts on $\CH^{\perp}$ fixed point freely, we may assume $\delta\in \CH$.

\begin{lmm}\label{order}
Under the above assumption, $|\tilde{g}_{|\bC E}| =R$. Here and there 
$\tilde{g}_{|A}$ denotes the restriction 
to $A$.  
In particular, 
${\rm Span}_{\bZ}\{m \mid U(m\alpha)_1\not=0\}+\bZ R=\bZ$. 
\end{lmm}  

\pr We first note that $N_0=1$ by the assumption.  Since $\tilde{g}$ acts on $U(m\alpha)\  (=U(mN_0\alpha))$ as a multiple by scalar $e^{2\pi im/R}$ and 
$\tilde{g}=\hat{\tau}\exp(\delta(0))$ with $\delta\in \CH$, 
we have $|\hat{\tau}_{|V_{E_{\tau}}}|=R$. 
Clearly, $\hat{\tau}_{|V_{E_{\tau}}}$ is a 
standard lift of $\tau_{|E_{\tau}}$. Since $\tau_{|E_{\tau}}$ acts on $E_{\tau}$ fixed point freely, $|\tau_{|E_{\tau}}|=|\hat{\tau}_{|V_{E_{\tau}}}|=R$ by 
Lemma \ref{Fixedpf}.  
Moreover, $E_{\tau}\subseteq \oplus_{m=0}^{R-1} U(m\alpha)$ and 
$\tau_{|E_{\tau}} =\hat{\tau}_{|\oplus_{m=1}^{R-1} U(m\alpha)_1}$; thus, we have  
${\rm Span}_{\bZ}\{m\mid U(m\alpha)_1\not=0\}+\bZ R=\bZ$ as desired.  
\prend

\subsection{Multiple constructions}\label{sec:multi}
In this subsection, we will consider VOAs obtained by multiple Lorentzian (or orbifold) constructions. Let $V$ be a VOA of CFT-type and $0\not=\CH$ a Cartan subalgebra of $V_1$. 
Let $\alpha \in \CH$ and define $\rho=\rho_\alpha=(\alpha,  1,\frac{\langle \alpha, \alpha\rangle}{2})\in \CH\oplus \bC\Pi_{1,1}$.  
By the discussion in Sections \ref{Sec:Lorentzian} and \ref{sec:3.3}, 
we can obtain a VOA
$$V^{[\alpha]}={\rm Comm}(M(\rho), V\otimes V_{\Pi_{1,1}})/P, $$
where $P=\bC[\rho_\alpha(-j):j>0]J+
({\bf 1}-e^{N\rho_\alpha})_{-1}J$, 
$J=\Comm(M(\rho,V\otimes V_{\Pi_{1,1}})$ and 
$\Comm(J, V\otimes V_{\Pi_{1,1}})=V_{\bZ N\rho}$ with $N\in \bZ$ and 
$V_{\bZ 0\rho}$ denotes $M(\rho)$.
Set 
$\widehat{\CH}^{[\alpha]}=\{(\beta,   0, \langle \beta,\alpha\rangle)\mid \beta\in \CH\}$ and denote 
$\widehat{\beta}=(\beta,  0, \langle \beta,\alpha\rangle)\in \widehat{\CH}^{[\alpha]}$ for $\beta\in \CH$. 
Then $\widehat{\CH}^{[\alpha]}\subseteq J$ and 
$\widehat{\CH}^{[\alpha]}\cap P=0$. Therefore, we may view 
$\widehat{\CH}^{[\alpha]}\subseteq V^{[\alpha]}$. 
Then we can construct another VOA $(V^{[\alpha]})^{[\widehat{\beta}]}$.

\begin{prop} \label{abcommute}
Let $V$ be a VOA of CFT-type and let 
$\CH$ be a Cartan subalgebra of $V_1$. 
Let $\alpha, \beta\in \CH$ such that 
$\langle \alpha, \beta\rangle \in \bZ$. 
Then $(V^{[\alpha]})^{[\widehat{\beta}]}\cong (V^{[\beta]})^{[\widehat{\alpha}]}$,
where 
$\widehat{\alpha}=(\alpha,0,\langle\beta,\alpha\rangle)\in V^{[\beta]}$.
\end{prop}

\pr 
In order to repeat Lorentzian construction twice, we consider 
a VA  $$\widehat{V}:=V\otimes V_{\Pi_{1,1}}\otimes V_{\Pi_{1,1}}.$$ 
We denote an element in $(\bZ \alpha+\bZ \beta)\oplus \Pi_{1,1}\oplus \Pi_{1,1}\subseteq \widehat{V}_1$ 
by $(\gamma:a,b,c,d)$, where 
\[
\langle (a,b,c,d), (a',b',c',d')\rangle =-ab'-ba'-cd'-dc' \quad  
\text{ and } \quad \gamma\in \bZ\alpha+\bZ\beta.
\] 

On the process of Lorentzian construction from $V$  
by $\alpha$, we first define an isotropic element 
$\rho_{\alpha}:=(\alpha:1, \frac{\langle \alpha,\alpha\rangle }{2})\in (V\otimes V_{\Pi_{1,1}})_1$ and 
$J_\alpha:= \Comm(M(\rho_{\alpha}), V\otimes V_{\Pi_{1,1}})$. We consider an ideal 
$P_\alpha :=\bC[\rho_\alpha(-j):j>0]J_\alpha+(\1-e^{N\rho_\alpha})_{-1}J_\alpha$ and 
define $V^{[\alpha]} = J_\alpha/P_\alpha$.    
We may identify $\rho_{\alpha}$ with $(\alpha:1,\frac{\langle\alpha,\alpha\rangle}{2},0,0)$. 
Then we can view $J_\alpha$ as $\Comm(M(\rho_{\alpha}, (0:0,0,\bZ,\bZ)), \widehat{V})$ and   $V^{[\alpha]}$ as  $\Comm(M(\rho_{\alpha}, (0:0,0,\bZ,\bZ)), \widehat{V})/ (\bC[\rho_\alpha(-j):j>0]J_\alpha+(\1-e^{N\rho_\alpha})_{-1}J_\alpha )$, where $\Comm(\Comm(M(\rho_\alpha),\widehat{V}),\widehat{V})=
V_{\bZ N\rho_\alpha}$ with $N\in \bN$.   
In this VOA, we can identify $\beta$ with $\widehat{\beta}=(\beta:0,\langle \alpha,\beta\rangle ,0,0)
\in \rho_{\alpha}^{\perp}$ and define an isotropic element  
$\rho_{\widehat{\beta}}=(\beta:0,\langle \alpha,\beta\rangle ,1,\frac{\langle \beta,\beta\rangle }{2}) .$ 
Then the Lorentzian construction gives a VOA  $(V^{[\alpha]})^{[\widehat{\beta}]}$. Note that $\mathrm{Span}\{ \rho_\alpha, \rho_{\widehat{\beta}}\}$ forms a totally isotropic subspace in $\bQ\otimes_\bZ (L \perp \Pi_{1,1}\perp \Pi_{1,1})$.

Set $J_{\alpha, \hat{\beta}}:= \Comm(M(\rho_{\alpha},\rho_{\widehat{\beta}}),\widehat{V})$. Then by the same arguments in Lemma \ref{ZJ}, $V_{\bZ N\rho_\alpha} \otimes V_{\bZ N' \rho_{\hat{\beta}}}$ is in the center of $J_{\alpha, \hat{\beta}}$, where $N'$ is the smallest positive integer such that $N\beta \in L$.  Let 
\[
P_{\alpha, \hat{\beta}}= (\bC[\rho_\alpha(-j):j>0]J_{\alpha, \hat{\beta}}+(\1-e^{N\rho_\alpha})_{-1}J_{\alpha, \hat{\beta}} + \bC[\rho_{\hat{\beta}}(-j):j>0]J_{\alpha, \hat{\beta}}+(\1-e^{N'\rho_{\hat{\beta}}})_{-1}J_{\alpha, \hat{\beta}}  ).
\] 
Then $(V^{[\alpha]})^{[\widehat{\beta}]}$ can be identified with $J_{\alpha, \hat{\beta}}/P_{\alpha, \hat{\beta}}$. 
 
Since $\langle \alpha, \beta\rangle \in \bZ$, $(0,\langle \alpha,\beta\rangle,1,0)$ and $(1,0,0,-\langle\alpha,\beta\rangle)$ are in $\Pi_{1,1}\perp \Pi_{1,1}$. Thus, we can decompose $\Pi_{1,1}\perp \Pi_{1,1}$ into 
\[ 
\Pi_{1,1}\perp \Pi_{1,1}= \{\bZ(0,\langle \alpha,\beta\rangle,1,0)+\bZ(0,0,0,1)\}\perp  
\{\bZ(1,0,0,-\langle\alpha,\beta\rangle)+\bZ(0,1,0,0)\}.
\] 
Using this decomposition,  we can define an isotropic element 
$\rho_{\beta}=(\beta:0,\langle\alpha,\beta\rangle,1,\frac{\langle\beta,\beta\rangle}{2})$ from $\beta$, 
which coincides with $\rho_{\widehat{\beta}}$, and construct  
$V^{[\beta]}$ using $\Comm(M(\rho_{\beta}), V\otimes V_{\Pi_{1,1}})$, 
where $\Pi_{1,1}=\{(0,a\langle\alpha,\beta\rangle,a,b):a,b\in \bZ\}$. 
In this VA, the corresponding element to $\alpha$ becomes 
$\widehat{\alpha}=(\alpha:0,0,0,\langle\alpha,\beta\rangle)\in (\rho_{\beta})^{\perp}.$  
Then we have an isotropic element 
$\rho_{\widehat{\alpha}}=(\alpha:0,0,0,\langle\alpha,\beta\rangle)
+(0,1,0,0,-\langle\alpha,\beta\rangle)+\frac{\langle\alpha,\alpha\rangle}{2}(0,0,1,0,0)
=(\alpha,1,\frac{\langle\alpha,\alpha\rangle}{2},0,0),$ which coincides with $\rho_{\alpha}$. 
Therefore, we have 
\[
(V^{[\beta]})^{[\widehat{\alpha}]} \cong \Comm(M(\rho_{\beta},\rho_{\widehat{\alpha}}), 
\widehat{V})/ P_{\beta, \hat{\alpha}}\cong \Comm(M(\rho_{\alpha}, \rho_{\widehat{\beta}}), 
\widehat{V})/ P_{\alpha, \hat{\beta}}\cong (V^{[\alpha]})^{[\widehat{\beta}]}, 
\]
as we desired. 
\prend

\begin{ntn}
We use $V^{[\alpha, \beta]}$ to denote $
	(V^{[\alpha]})^{[\widehat{\beta}]}$ when $\langle \alpha,\beta\rangle\in \bZ$. 
\end{ntn}

\section{Automorphisms associated with W-elements} \label{sec:4}
\subsection{W-elements}
From now on, we will consider some special inner automorphisms defined by Weyl vectors. 
For a decomposition $V_1=\oplus_{j=1}^t \CG_{j,k_j}$ with simple 
Lie algebras $\CG_{j,k_j}$ with levels $k_j$, dual Coxeter numbers $h^{\vee}_j$ and lacing numbers $r_j$,  we choose a Weyl vector $\rho_j$ for each $\CG_{j,k_j}$ and 
define  
\begin{equation}\label{eq:5.1}
\alpha=\sum_{j=1}^t \frac{\rho_j}{h_j^{\vee}} \in V_1.   
\end{equation}

Set $g=\exp(2\pi  i\alpha(0))$. Note that $\alpha$ is uniquely determined if we fix a set of simple roots (or a positive root system). Moreover,  
$\langle \rho_j,\beta\rangle=1$ for a long simple root $\beta$ and $\langle \rho_j, \beta\rangle=\frac{1}{r_j}$ 
for a short simple root $\beta$. 
In this paper, the element $\alpha$ defined by Weyl vectors plays a very important role. We give it a name and call it a $W$-element.   
Since $g$ acts on an affine vertex algebra ${\rm VA}(V_1)$ as an automorphism 
of finite order $T={\rm LCM}(r_1h_1^{\vee},...,r_th_t^{\vee})$ and 
$V$ is a finite extension of ${\rm VA}(V_1)$-module, 
$g$ has a finite order on $V$. 
Let $N$ be the order of $g=\exp(2\pi i\alpha(0))$.  

\begin{lmm} $T$ divides $N$.  In particular, $h_j^{\vee}|N$. 
The set of eigenvalues of $N\alpha(0)$ on $V$ is $\bZ$. 
\end{lmm}

\pr 
Since $\alpha_{j}(0)=\rho_{j}(0)/h_j^{\vee}$ has an eigenvalue $1/r_jh_j^{\vee}$ on $\CG_j$, it is clear that $T={\rm LCM}(h^{\vee}_1r_1,\cdots,h^{\vee}_tr_t)$ divides $N$ and $S$ divides
$N/T$, where ${\rm Span}_{\bZ}\{\mbox{eigenvalues of }\alpha(0)\}=\bZ \frac{S}N$.  Since $(N,S)=1$, we have $S=1$ as desired. 
\prend

By Proposition \ref{prop:2.3} and Theorem \ref{thm:2.4} in \S 2, we have: 

\begin{prop}\label{dimV1}
By setting $\langle \alpha,\alpha\rangle=\frac{2K_0}{N_0}$ with $(K_0,N_0)=1$, 
we have 
\begin{equation}\label{eq:5.2}
\frac{\dim V_1}{\dim V_1-24}=\frac{K_0}{N_0},  \quad \frac{k_j}{h_j^{\vee}}=\frac{K_0-N_0}{N_0}, \quad \mbox{and} \quad 
\dim V_1=\frac{24K_0}{K_0-N_0}. 
\end{equation}
In particular, $K_0>N_0$, $(K_0-N_0)|24$ and $N_0|h_j^{\vee}$ and 
$(K_0-N_0)|k_j$ for all $j$. 
\end{prop}

\pr
By direct calculations, we have
\[
\begin{split}
\langle \alpha,\alpha\rangle=&\langle \sum \frac{\rho_j}{h^{\vee}_j}, 
\sum\frac{\rho_j}{h^{\vee}_j}\rangle=\sum \frac{\langle \rho_j,\rho_j\rangle}{(h^{\vee}_j)^2}=\sum \frac{h_j^{\vee}k_j\dim \CG_j}{12 (h_j^{\vee})^2}=\sum \frac{k_j\dim \CG_j}{12 h_j^{\vee}}\\
=&\sum_j \frac{24\dim \CG_j}{12(\dim V_1-24)} 
=\frac{2\dim V_1}{\dim V_1-24}. 
\end{split}
\]
Therefore, 
\[
\dim V_1=\frac{24\langle \alpha,\alpha\rangle}{
\langle \alpha,\alpha\rangle-2}= \frac{24K_0}{K_0-N_0}
\]
and so $(K_0-N_0)|24$ since $(K_0,K_0-N_0)=1$. We also have: 
\[ 
\frac{h_j^{\vee}}{k_j}=\frac{1}{24}\left( \frac{24\langle\alpha,\alpha\rangle}{\langle \alpha,\alpha\rangle-2}-24\right)=\frac{2}{\langle 
\alpha,\alpha\rangle-2}= \frac{N_0}{K_0-N_0}
\]  
and so $N_0|h_j^{\vee}$ since $(N_0,K_0-N_0)=1$. 
\prend

\begin{ntn}
If $\alpha=\sum_{j=1}^t \frac{\rho_j}{h_j^{\vee}}$ is a W-element,  we often denote the VOA $V^{[\alpha]}$ by $\widetilde{V}$. 
\end{ntn}

The following result is very important. 

\begin{prop} \label{N0a}
We have $N_0\alpha\in L^{\ast}$. 
In particular, $\Comm(M(\CH),\tilde{V})=\displaystyle{\bigoplus_{m=0}^{R-1}U(mN_0\alpha)}$.  
\end{prop}
\pr 
	Let  $V_1=\CG=\oplus_{j=1}^t {\CG}_{j,k_j}$. 
Recall that the lattice $L$ contains a sublattice  isometric to $\mathfrak{Q}= \oplus_{j=1}^r \sqrt{k_j} Q_{j}^{\ell}$, where $Q_{j}^{\ell}$ is the sublattice generated by long roots of $\CG_j$. 
Note that it is also the sublattice generated by the co-roots, i.e., $2\beta/(\beta, \beta)$, $\beta$ a root, where $(\ , \ )$ is the normalized Killing form. 

Let  $\bigotimes_{i=1}^t L_{\hat\CG_i}(k_i,\lambda_i) <V$ be a simple current module of $\bigotimes_{i=1}^t L_{\hat\CG_i}(k_i,0)$.   Then  
\[
\lambda=(\frac{1}{\sqrt{k_1}} \lambda_1,\dots, \frac{1}{\sqrt{k_t}} \lambda_t  ) \in L.
\]
The lattice $L$ is indeed spanned by  $\mathfrak{Q}$ and $(\frac{1}{\sqrt{k_1}} \lambda_1,\dots, \frac{1}{\sqrt{k_t}} \lambda_t  )$ for all $(\lambda_1, \dots, \lambda_t)$ satisfying the above condition.

For any root $\beta \in Q_i^{\ell}$,  we have $(\alpha, \sqrt{k_i} \beta)= k_i/h_i^{\vee}= (K_0-N_0)/N_0$. Therefore, 
$(N_0\alpha , \gamma) \in \bZ $ for any $\gamma\in \mathfrak{Q}$. Note that $\alpha= \sum \sqrt{k_i}\rho_i/h_i^{\vee}$.

Now let $\bigotimes_{i=1}^tL_{\hat\CG_i}(k_i,\lambda_i)<V$ be a simple current module of $\bigotimes_{i=1}^t L_{\hat\CG_i}(k_i,0)$.  Then the conformal weight of the module $\bigotimes_{i=1}^t L_{\hat\CG_i}(k_i,\lambda_i)$ must be an integer and is given by 
\begin{align*}
	\sum_{i=1}^t\frac{(2\rho_i+\lambda_i,\lambda_i)}{2(k_i+h_i^\vee)}&=\frac{h_i^\vee}{2(h_i^\vee+k_i)}\sum_{i=1}^t\left(2(\rho_i/h_i^\vee,\lambda_i)+\frac{1}{h_i^\vee}
	(\lambda_i,\lambda_i)\right)\\
	&=\frac{N_0}{2K_0}\sum_{i=1}^r\left(2(\rho_i/h_i^\vee,\lambda_i)+\frac{K_0-N_0}{N_0}
	(\frac{1}{\sqrt{k_i}}\lambda_i,\frac{1}{\sqrt{k_i}}\lambda_i)\right),\\
	&= \frac{1}{2K_0}\left( 2(N_0\alpha, \lambda)+ (K_0-N_0)(\lambda, \lambda)\right).
\end{align*}

Then $2(N_0\alpha, \lambda)+ (K_0-N_0)(\lambda, \lambda)$ must be an even integer but $\lambda\in L$ and we have $(\lambda, \lambda)\in 2\mathbb{Z}$ and hence  $2(N_0\alpha, \lambda)\in 2\mathbb{Z}$ and $(N_0\alpha, \lambda)\in \mathbb{Z}$ as desired. 

By Theorem \ref{Lalpha}, we have  $\Comm(M(\CH),\tilde{V})=\oplus_{m=0}^{R-1}U(mN_0\alpha)$. 
  \prend  

\medskip

Set $R=N/N_0\in \bN$. 
We note that $A={\rm LCM}(r_1\frac{h^{\vee}_1}{N_0},...,
r_t\frac{h^{\vee}_t}{N_0})$ divides $R$ and $K=RK_0$. 

\begin{prop}\label{NoOne}
If $A=1$ or $R=1$, then $V=V_{E}$ for some Niemeier lattice $E$. 
\end{prop}
 
\pr 
If $A=1$ or $R=1$, then $r_j=1$ and $h_j^{\vee}=N_0$ and $k_j=K_0-N_0$ for all $j$.  
In particular, $(k_j,h_j^{\vee})=1$ and $k_j|24$ by Proposition \ref{dimV1}. 
Denote $h_j^{\vee}$ and $k_j$ by $h^{\vee}$ and $k$, respectively. 
Since $\CG_j$ are all simply laced, 
$\dim V_1=(h+1)n$, where $n={\rm rank}(V_1)$.  
Therefore, we have  $ kn(h+1)=24(k+h)$. 
If $k=1$, then $n=24$ and $V$ is a Niemeier lattice VOA.  
So we may assume $k\geq 2$. 
If $k$ is even, say $k=2t$, then $h$ is odd. A simply laced simple Lie algebra with odd (dual) Coxeter number must be of type $A$. Therefore, the simple components for $V_1$ are all of type $A_{h-1}$ and hence $n=s(h-1)$ for some $s\in \bN$. Then 
$t\times s(h-1)(h+1)=12(2t+h)$.  However, $8|(h-1)(h+1)$, but $8\not | 12(2t+h)$, which is a contradiction. 
If $k$ is odd, then since $k|24$, $k=3$ and so $n(h+1)=8(h+3)$.   Furthermore, if $h$ is even, then $h+1$ divides $h+3$, which is a contradiction. If $h=2s-1$, then $sn=8(s+1)$ and so $s=2,4$ or $8$ since 
$h\geq 2$. 
But in these cases, we have $h=3,7,15$. 
Since $(k,h)=1$, we have $h=7$. 
Hence $s=4$ and so $4n=8\times 5$, that is, $n=10$. 
On the other hand, since $h=7$, we have $\CG_j=A_6$, which contradicts $n=10$. 
\prend

\begin{cry}\label{ToLeech}
For any holomorphic strongly regular VOA $V$ of CFT-type with central charge $24$, 
by repeating the orbifold constructions of inner automorphism defined by W-elements (see \eqref{eq:5.1}), 
we can construct the Leech lattice VOA.
 \end{cry}

\pr 
For any holomorphic VOA $V$ of central charge $24$, we know that 
$\dim \CH\leq 24$ and $U$ is $C_2$-cofinite.  
By Corollary \ref{Tilde} and Proposition \ref{NoOne}, if we repeat the orbifold constructions 
by $W$-elements, then we will arrive at a Niemeier lattice VOA. 
Furthermore, we have already shown that we can construct 
the Leech lattice VOA from Niemeier lattice VOAs by a $W$-element.  
\prend

\begin{thm}\label{OverQ} If $V$ is a strongly regular holomorphic VOA of central charge $24$ and 
$V_1\not=0$, then $V$ is realizable over $\bQ$. 
\end{thm}

\pr
Since $V_{\Lambda}$ is realizable over $\bQ$, 
we may assume that $\tilde{V}=V^{[\alpha]}$ is realizable over $\bQ$ by induction.  
Since $\tau$ is given by a permutation of roots in $\tilde{V}_1$ by the above 
proposition or $\tilde{V}=V_{\Lambda}$ and $\tau\in Co_0$, 
it is also an automorphism of $\tilde{V}^{\bQ}$. 
Furthermore, since $e^{N_0\widehat{\alpha}}\in {\rm Comm}({\rm Comm}(M(\tilde{\CH}),\tilde{V}), \tilde{V})=V_{Q}$ for some 
even lattice $Q$, we may assume $\widehat{\alpha}\in \tilde{V}^{\bQ}$.  Therefore, a $\widetilde{g}^m$-twisted $\tilde{V}^{\bQ}$-module is also 
realizable over $\bQ$ by \cite{DLM} for every $m$. 
Since $V$ is constructed from $\tilde{V}$ by orbifold 
construction with an automorphism $\widetilde{g}$, we get that $V$ is also realizable over $\bQ$. 
\prend

Actually, our desired Theorem \ref{Main} asserts that we can obtain $V_{\Lambda}$ directly by a single orbifold construction. In order to prove it, we will consider $V$ which gives $V_{\Lambda}$ in two steps. We first introduce a partial order $>\!\!>$ on the set of holomorphic VOAs of central charge $24$.  
Namely, $V'>\!\!>V$ if (1) $\dim\CH'>\dim \CH$ or (2) $\dim \CH'=\dim \CH$ and 
$\dim\Comm(M(\CH'),V')_m\geq \dim\Comm(M(\CH),V)_m$ for all $m$ and 
$\dim\Comm(M(\CH'),V')_m> \dim\Comm(M(\CH),V)_m$ for some $m$,
where $\CH'$ is a Cartan subalgebra of $V'_1$. 

\medskip

Set $\CF^1=\{ V\mid V=V_{\Lambda}\mbox{ or }\widetilde{V}\cong V_{\Lambda} \}$ and  
$$\CF^2=\CF^1\cup \{ V \mid V'\in \CF^1  
\mbox{ for any VOA $V'>\!\!>V$ } \}$$
In order to prove our main theorem, it is sufficient to show that 
$\CF^2=\CF^1$.

\medskip

From now on, we will treat only $V$ in $\CF^2$.  
We note that $\widetilde{V}$ is in $\CF^1\subseteq \CF^2$.


\medskip

By using the commutativity of multi-constructions, we have:

\begin{cry}
	Let $\alpha$ be a W-element and $\langle\alpha,\alpha\rangle =2K_0/N_0$ and $t\in \bZ$. 
	If we do an orbifold construction $V^{[tN_0\alpha]}$ from $V$ by 
	$g^{tN_0}=\exp(2\pi itN_0\alpha(0) )$ 
	and then we do an orbifold construction $(V^{[tN_0\alpha]})^{[\alpha]}$ from 
	$V^{[tN_0\alpha]}$ by $g=\exp(2\pi i\alpha(0))$, then we have 
	$(V^{[tN_0\alpha]})^{[\alpha]}\cong V^{[\alpha]}$, where we identify $\alpha$ and 
	its corresponding element in $V^{[tN_0\alpha]}$.
\end{cry}

\pr
Set $\Comm(\Comm(M(\CH),V^{[\alpha]}),V^{[\alpha]})=V_{\tilde{L}}.$ 
We have shown that $N_0\alpha\in \tilde{L}$. 
By Proposition \ref{abcommute}, we have 
$(V^{[tN_0\alpha]})^{[\alpha]}\cong (V^{[\alpha]})^{[tN_0\alpha]}.$ 
On the other hand, since we have $N_0\alpha\in \tilde{L}$,   
$\exp(2\pi itN_0\alpha)$ is trivial on $V^{[\alpha]}$ 
and so  $(V^{[\alpha]})^{[tN_0\alpha]}\cong V^{[\alpha]}$.
\prend

Next we study the structure of $V^{[\alpha, \beta]}$ and a commutant 
$U^{[\alpha, \beta]}:=\Comm(M(\CH),V^{[\alpha,\beta]})$.  
For simplicity, we only consider the following cases. 

Let $\alpha',\beta'\in \CH$ such that $\bZ\alpha'+\bZ\beta'+L\subseteq L^{\ast}$ is an even sublattice.  Our main example for $\bZ\alpha'+\bZ\beta'+L$ will be $\bZ N_0\alpha+\bZ \tilde{N}_0\beta+L\subseteq L^{\ast}$ for W-elements $\alpha$ of $V_1$ and $\beta$ of $\tilde{V}_1$ (cf. Section \ref{sec:4}). 

Let $R'$ be the least natural integer such that $R'\alpha'\in L$. Then $R'$ is the order of $\exp(2\pi i\alpha'(0))$.  
We first construct $V^{[\alpha']}$. 
Set $U^{[\alpha']}:=\Comm(M(\CH),V^{[\alpha']})$. Since $\langle \alpha',\alpha'\rangle\in 2\bZ$ and $\alpha'\in L^*$,   
$U^{[\alpha']}=\oplus_{m=0}^{R'-1} U(m \alpha')$  by Theorem \ref{Lalpha}. Moreover, $\Comm(U^{[\alpha']},V^{[\alpha']})=
V_{\tilde{L}}$, where $\tilde{L}=L+\bZ\alpha'$.   

Viewing $V^{[\alpha']}$ as a $U^{[\alpha']}\otimes V_{\tilde{L}}$-module, we have 
\[
\begin{split}
	V^{[\alpha']}=&\bigoplus_{\delta\in \tilde{L}^{\ast}/\tilde{L}}U^{[\alpha']}(\delta)
	\otimes V_{\tilde{L}+\delta}\\
	=&\bigoplus_{\delta\in (L+\bZ\alpha')^{\ast}/(L+\bZ\alpha')}\{
	\oplus_{m=0}^{R'-1} U(\delta+m\alpha')\}\otimes \{\oplus_{n=0}^{R'-1} V_{\delta+n\alpha'+L}\}.
\end{split} 
\]
Since $\beta'\in \bQ L=\bQ(L+\bZ\alpha')$, 
we can construct $V^{[\alpha',\beta']}$ from $V^{[\alpha']}$ by an orbifold construction with $\exp(2\pi i\beta'(0))$.  
In this case, 
$$\begin{array}{rl}
	U^{[\alpha', \beta']}:=&\Comm(M(\CH),V^{[\alpha',\beta']})\cr
	=&\oplus_{n=0}^{\tilde{R}'-1}U^{[\alpha']}(n\beta') \cr
	=&\oplus_{n=0}^{\tilde{R}'-1}\oplus_{m=0}^{R-1}U(n\beta'+m\alpha') 
\end{array}$$
where $\tilde{R}'$ is the least natural integer satisfying $\tilde{R}'\beta\in L+\bZ\alpha'$ and   
$$\Comm(U^{[\alpha', \beta']},V^{[\alpha',\beta']})=V_{L+\bZ\alpha'+\bZ\beta'}.  $$
Set $L(\alpha', \beta')=L+\bZ\alpha'+\bZ\beta'$. 
Viewing $V^{[\alpha',\beta']}$ 
as $U^{[\alpha', \beta']}\otimes V_{L(\alpha', \beta')}$-modules, we decompose 
$$ V^{[\alpha',\beta']}=\bigoplus_{\gamma\in L(\alpha', \beta')^{\ast}/L(\alpha', \beta')} U^{[\alpha', \beta']}(\gamma)\otimes V_{L(\alpha', \beta')+\gamma}$$
and so 
$$
V^{[\alpha',\beta']}=\bigoplus_{\gamma\in L(\alpha', \beta')^{\ast}/L(\alpha', \beta')} 
\{\oplus_{m=0}^{R'}\oplus_{n=0}^{\tilde{R}'}U(\gamma+m\alpha'+n\beta')\}
\otimes\{\oplus_{m=0}^{R'}\oplus_{n=0}^{\tilde{R}'} V_{\gamma+m\alpha'+n\beta'+L}\}.$$
From the above explanation, it is easy to see 
$V^{[\alpha',\beta']}=V^{[\alpha',\beta'+p\alpha']}$ for any $p\in \bZ$ and 
$V^{[\alpha',\beta']}=V^{[\beta',\alpha']}$. Namely, the VOA $V^{[\alpha',\beta']}$ depends only on the lattice $\bZ\alpha'+\bZ\beta'+L$ and we have the following useful result:

\begin{prop}\label{abLindep}
	Let $\alpha, \beta, \alpha', \beta'\in L^*$ such that $\bZ\alpha+\bZ\beta+L$ and $\bZ\alpha'+\bZ\beta'+L$ are even lattices. Suppose $\bZ\alpha+\bZ\beta+L=\bZ\alpha'+\bZ\beta'+L$. 
	Then $V^{[\alpha, \beta]} \cong  V^{[\alpha', \beta']}$.   
\end{prop}

\subsection{Proof of Proposition 1.2}
In this subsection, we will prove Proposition 1.2. Recall from our hypothesis that $\alpha$ is a $W$-element of $V_1$ with   $\langle \alpha, \alpha \rangle =2K_0/N_0$ and $(K_0, N_0)=1$. We also assume $V^{[\alpha]}\cong V_\Lambda$.  
Let $g=\exp(2\pi i \alpha(0))$ and let $N$ be the order of $g$ in $V$. Set $R=N/N_0$ and assume $(K_0,R)=1$. (The proof for the case $(N_0,R)=1$ will be similar and the notation is simpler.) 
Then there are $a,b\in \bZ$ such that $aK_0+bR=1$. To simplify the notation, set $t=N_0/K_0$.
As we explained in Lemma \ref{Dual}, $V^{[t\alpha]}\cong V^{[\alpha]}\cong V_{\Lambda}$. 
Since $N_0\alpha\in \tilde{L}$,  $\exp(2\pi iN_0\alpha(0))=1$ on $V^{[t\alpha]}$. 
Therefore, 
$V_{\Lambda}\cong V^{[t\alpha,N_0\alpha]}\cong V^{[N_0\alpha,t\alpha]}
=V^{[N_0\alpha, aK_0t\alpha+bRt\alpha]}=
V^{[N_0\alpha,aN_0\alpha+bRt\alpha]}=V^{[N_0\alpha,btR\alpha]}
=V^{[btR\alpha,N_0\alpha]}$ since $\langle t\alpha, N_0\alpha\rangle\in \bZ$. 

Set $s= tbR$ and $g_s=\exp(2\pi itbR\alpha(0))$. Since 
$(bN_0R,K_0)=1$, $mtbR\alpha=m\frac{bN_0R}{K_0}\alpha\in L$ if and only if $K_0|m$, that is, $|g_s|=K_0$. 
Moreover, $\langle bN\alpha/K_0, bN\alpha/K_0\rangle=\frac{2b^2R^2N_0}{K_0}$ 
with $(K_0,b^2R^2N_0)=1$.   
Therefore, 
$\Comm(M(\CH),V^{[bN\alpha/K_0]})=U(0)$ and so 
$\rank(V^{[bN_0\alpha/K_0]}_1)=\rank(V_1)$. 
Furthermore, since $|g^{N_0}|=R$ and 
$(V^{[bN\alpha/K_0]}_1)^{<g^{N_0}>}=\CH$, 
the Coxeter number of $V^{[bN\alpha/K_0]}_1$ is 
less than or equal to $R$ (cf. \cite[Execrise 8.11]{kac}). 

For Statement (3) in Proposition 1.2, we note that 
$\langle btR\alpha, btR\alpha\rangle=\frac{2b^2R^2N_0}{K_0}$ 
with $(K_0,b^2R^2N_0)=1$ and 
$|\exp(2\pi ibN\alpha(0)/K_0)|=K_0$. Therefore, $|g_s|/K_0=1$ and   
we have $\widetilde{\tilde{V}(g_s)}(g_s)=(V^{[btR\alpha]})^{[btR\alpha]}=V$ by  Proposition \ref{R=1}.    
 
This completes the proof of Proposition 1.2.

\begin{cry}\label{No3C}
Let $V$ be a holomorphic VOA. Let $\CH$ be a Cartan subalgebra of $V_1$ and $\Comm(\Comm(M(\CH),V),V)\cong V_L$. 
Let $\alpha\in L^{\ast}$ and set $g=\exp(2\pi i\alpha(0))$. 
Assume $N_0\alpha\in L^{\ast}$ and $\tilde{V}(g)=V_{\Lambda}$.  
Let $\tilde{g}$ be a reverse automorphism of $\tilde{V}(g)$ for orbifold construction 
by $g$ and set $\tau=\tilde{g}_{|\bC \Lambda}\in O(\Lambda)=Co_0$. 
Then $\tau\not=-2A$, $3C$, nor $5C$. 
\end{cry}

\pr 
Suppose false. In particular, $\tau\not=1$ and so $\rank(V_1)\not=24$.
Since $\tilde{V}(g)=V^{[\alpha]}=V_{\Lambda}$, we have 
$\oplus_{m=0}^{R-1} U(mN_0\alpha)_1+\CH=\bC\Lambda$ and so 
$\rank(V^{[N_0\alpha]}_1)=24$, that is,  
$V^{[N_0\alpha]}=V_E$ for some Niemeier lattice $E$. 
Since $N_0\alpha\in L^{\ast}$ and $\langle N_0\alpha,N_0\alpha\rangle=2K_0N_0\in 2\bZ$, 
we have $R=|\tau|$ by Lemma \ref{order}.  
Since $(N_0,K_0)=1$ and $R$ is a prime number by the assumption, 
we have $(R,N_0)=1$ or $(R,K_0)=1$ and so $V$ satisfies the conditions in Proposition 1.2. 
Therefore, there is a holomorphic VOA $V'$ satisfying $\rank V'_1=\rank V_1$ and 
Coxeter number of components of $V'_1$ is less than or equal to $R$. 
Hence $\dim V'_1\leq (R+1)\rank(V'_1)$. 
Since $\dim\bC\Lambda^{<-2A>}=8, \dim\bC\Lambda^{<3C>}=6, \dim\bC\Lambda^{<5C>}=4$, 
we have $\dim V'_1\leq 3\times 8, \dim V'_1\leq 4\times 6, \dim V'_1\leq 6\times 4$, 
respectively, which contradicts $\dim V'_1>24$. 
\prend

\section{Roots in Twisted Modules}

\subsection{$g$-twisted module}

\begin{lmm}\label{inn-1}
Let $\CG$ be a simple Lie algebra, $\Phi$ the set of 
all roots and $\rho$ a Weyl vector. Then  
$\langle \rho/h^{\vee}, \beta\rangle > -1$ for all $\beta\in \Phi$, where $h^{\vee}$ is the dual Coxeter number. 
\end{lmm}

\pr 
Let $\theta=\sum a_i\beta_i$ be a highest root. 
Then $\langle \rho, \theta\rangle=h^{\vee}-1$. 
Let $\beta=\sum b_i\beta_i\in \Phi^+$. Then $0\leq b_i\leq a_i$ and some $b_i\not=0$. 
Therefore, $0 \lneq \langle \rho/h^{\vee},\beta\rangle<
\frac{h^{\vee}-1}{h^{\vee}}<1$. 
Therefore, for any root $\beta\in \Phi$, we have 
$ -1 < \langle \rho/h^{\vee},\beta\rangle<1$ as we desired. 
\prend

We will next prove that $g$-twisted module $T^1$ has no elements with 
weight less than or equal to one except for elements in $\widehat{\CH}^{\perp}$. 

\begin{thm}\label{twist} 
For $bN\not=K$,  $(X(\frac{(K+bN)\alpha}{2K})\otimes e^{\frac{(-K+bN)\alpha}{2K}})_1=0$.  In particular, $T^1_1\subseteq \CH^{\perp}$.
\end{thm}

First we recall Lemma 3.4 in \cite{LS} and 
its proof with a slight  change of notation. 
For a $V$-module $M$, let $M^{(\alpha)}$  be defined as in Proposition \ref{Delta_twist}  
and $L^{(\alpha)}(0)$ the grading operator $L(0)$ on 
$M^{(\alpha)}$.

\begin{lmm}
Let $\alpha= \sum_{i}^t \rho_i/h_i^{\vee}$ be a W-element and 
$M$ a $L_{\CG_1}(k_1,0)\otimes \cdots\otimes L_{\CG_t}(k_t,0)$-module. 
Let $v$ be a vector in $M^{(-\alpha)}$ with $L^{(-\alpha)}(0)$-weight $p$, i.e. 
$L^{(-\alpha)}(0)v=pv$. Let $u={E_{\beta_i}}_{(-n_1)}\cdots {E_{\beta_m}}_{(-n_m)}v\in M^{(-\alpha)}$ be a non-zero vector, where $n_i\in \bZ_{>0}$ and $E_{\beta_i}\in V_1$ denotes a root vector associated with a root $\beta_i\in \Phi$. 
Then $u$ is a homogeneous vector in $M^{(-\alpha)}$ and its 
$L^{(-\alpha)}(0)$-weight is greater 
than $p$. 
\end{lmm}

\pr  By Haisheng Li's $\Delta$-operator, 
$L^{(-\alpha)}(0)=L(0)-\alpha_{(0)}+\frac{\langle \alpha,\alpha\rangle}{2}{\rm id}$. Since $[L^{(-\alpha)}(0), E_{\beta(-n)}]=(n-(\alpha|\beta)){\rm id}$, 
we have 
$$ L^{(-\alpha)}(0)u=\left( \sum_{i=1}^m (n_i-(\alpha|\beta_i))+p \right) u.$$ 
By Lemma \ref{inn-1},    
  $1\gneq \langle \alpha,\beta\rangle \gneq -1$ 
  for all $\beta\in \Phi$,
     we have 
  $n_i-(\alpha|\beta_i)> 0$ for all $i$, and 
  hence the $L^{(-\alpha)}(0)$-weight of $u$ in $M^{(-\alpha)}$ is greater than $p$.
\prend

We start the proof of Theorem \ref{twist}. 
Let $V$ be a holomorphic VOA of central charge $24$. 
Suppose that $V_1\cong \CG_{1,k_1}\oplus \cdots 
\oplus \CG_{t,k_t}$ is semisimple. Denote 
$$L(\lambda_1,...,\lambda_n)=L_{\CG_1}(k_1,\lambda_1)\otimes 
\cdots \otimes L_{\CG_t}(k_t,\lambda_t),$$ 
where $(\lambda_1,...,\lambda_t)$ is an integrable weight of $\CG_{1,k_1}\oplus \cdots 
\oplus \CG_{n,k_t}$. 

Since $V=\oplus L(\lambda_1,...,\lambda_t)$, by Haisheng Li's $\Delta$-operators, $T^1=\oplus L(\Lambda_1,...,\lambda_t)^{(-\alpha)}$, 
where $L(\lambda_1,...,\lambda_t)^{(-\alpha)}$ denotes the one transformed 
from $L(\lambda_1,...,\lambda_t)$.

\begin{lmm}\label{minweight}
Let $\CG$ be a simple (finite dimensional) Lie algebra. Let $\{\alpha_1,\dots, \alpha_n\}$ be a set of simple roots and $\rho$ a Weyl vector. Let $\lambda$ be a dominant integrable weight of $\CG$. Then 
\[
- (\rho|\lambda) \leq (\rho| \mu)\leq  (\rho|\lambda) \quad \text{ for any }  \mu \in \Pi(\lambda,\CG), 
\]
where $\Pi(\lambda,\CG)$ denotes the set of all weights of the irreducible module of $\CG$ with the  highest weight $\lambda$. 
\end{lmm}

\pr 
Since $\lambda$ is a highest weight, every weight $\mu \in \Pi(\lambda,\CG)$ can be written 
as $$\mu =\lambda -\sum a_i \alpha_i,$$ where $a_i$'s are non-negative integers. Then 
$(\rho| \mu)=(\rho|\lambda) - \sum a_i (\rho|\alpha_i) \leq  (\rho|\lambda)$ as desired. 

Let $w_0$ be the longest element in the Weyl group. Then $w_0(\lambda)$ is a lowest weight and 
we have 
\[
(\rho| \mu) \geq  (\rho|w_0(\lambda)). 
\]
Note that either the longest element $w_0$ is  equal to $-1$ or  $-w_0$ is a diagram automorphism.
Since $\rho$ is fixed by all diagram automorphisms, $(\rho|w_0(\lambda)) = - (\rho| \lambda)$ and we have the desired result. 
\prend

\begin{lmm} \label{T101}
The conformal weight of $L(\lambda_1,...,\lambda_t)^{(-\alpha)}$ is 
$\geq 1$ and is $1$ if and only if $\lambda_i=k_i\rho_i/h_i^{\vee}$ for all $i$, 
where $(\cdot)^{(-\alpha)}$ denotes $g$-twisted structure.  
\end{lmm}

\pr 
By \cite[Lemma 3.6]{LS}, the conformal weight of $L(\lambda_1, \dots, \lambda_t)^{(-\alpha)}$ is given by 
\[
\begin{split}
w= \sum_{i=1}^t \left( \frac{(\lambda_i| 2\rho_i+\lambda_i)}{2(k_i+h_i^\vee)} + \frac{1}{h_i^\vee} 
\mathrm{min}\{(-\rho_i|\mu)\mid \mu\in\Pi(\lambda_i,\CG_i)\} + \frac{k_i}{2(h_i^\vee)^2}(\rho_i|\rho_i)\right).
\end{split}
\]
Then by Lemma \ref{minweight}, we  have
\[
\begin{split}
w&= \sum_{i=1}^t \left( \frac{(\lambda_i| 2\rho_i+\lambda_i)}{2(k_i+h_i^\vee)} - \frac{1}{h_i^\vee} (\rho_i, \lambda_i)+ \frac{k_i}{2(h_i^\vee)^2}(\rho_i|\rho_i)\right),\\
&= \sum_{i=1}^t\frac{1}{2(h_i^\vee)^2(k_i+h_i^\vee)} \left[ (h_i^\vee)^2 (\lambda_i| 2\rho_i+\lambda_i) - 2h_i^\vee(k_i+h_i^\vee)
(\rho_i, \lambda_i) + k_i(k_i+h_i^\vee) (\rho_i|\rho_i) \right], \\
&= \sum_{i=1}^t \frac{1}{2(h_i^\vee)^2(k_i+h_i^\vee)} \left[ (h_i^\vee\lambda_i -k_i\rho_i| h_i^\vee\lambda_i -k_i\rho_i) 
+ k_ih_i^\vee(\rho_i|\rho_i)\right],\\
&\geq  \sum_{i=1}^t \frac{1}{2(h_i^\vee)^2(k_i+h_i^\vee)} \left[ k_ih_i^\vee(\rho_i|\rho_i)\right]
\end{split}
\] 
and the equality holds if and only if  $h_i^\vee\lambda_i -k_i\rho_i=0$ for all $i$. 

Moreover, 
\[
\begin{split}
 \sum_{i=1}^t \frac{1}{2(h_i^\vee)^2(k_i+h_i^\vee)} \left[ k_ih_i^\vee(\rho_i|\rho_i)\right]
= \sum_{i=1}^t \frac{k_i(\rho_i|\rho_i)}{2(h_i^\vee)^2} \cdot \frac{h_i^\vee}{(k_i+h_i^\vee)}.
\end{split}
\]
Recall that $\sum_{i=1}^n \frac{k_i(\rho_i|\rho_i)}{2(h_i^\vee)^2}= \frac{\dim V_1}{\dim V_1- 24}$ (cf. Proposition \ref{dimV1})  and 
\[
\frac{h_i^\vee}{(k_i+h_i^\vee)} = \frac{1}{( k_i/h_i^\vee+1)} = \frac{\dim V_1-24}{\dim V_1},
\]
which is independent of $i$ by \cite{DM2}. 
Thus, we have $w\geq 1$ and $w=1$ if and only if  $\lambda_i =(k_i/h_i^\vee) \rho_i$ for all $i$.
\prend

Namely, the weight one space of $T^1$ comes from 
$$L_{\CG_1}(k_1, \frac{K-N}{N}\rho_1)\oplus \cdots \oplus
L_{\CG_t}(k_t, \frac{K-N}{N}\rho_t)$$
and all weight one elements in $T^1$ centralize $\CH$. Thus we also have the following theorem.

\begin{thm}\label{Thm:5.6}
$T^1_s=0$ for $s<1$ and $T^1_1\subseteq \widehat{\CH}^{\perp}$.  Moreover, 
$T_1^1=0 $ if $N_0\neq 1$
\end{thm}
\begin{proof}
As we have shown, in our Lorentzian notation (cf. Section \ref{sec:3}), there is a $V^{<g>}$-isomorphism 
$\phi_1: T^{1,0} \to \overline{Z^{1,0}} \subseteq \tilde{V}$ which shifts 
$$\bigoplus_{m=0}^{2K-1} M(\CH)\otimes U(\frac{(K+mN)\alpha}{2K}+u)
\otimes e^{\frac{(K+mN)\alpha}{2K}+u}  \subseteq V \, (\cong T^{1} \text{ as a vector space}) 
$$
to  
$$\bigoplus_{m=0}^{2K-1} M(\widehat{\CH})\otimes U(\frac{(K+mN)\alpha}{2K}+u)
\otimes e^{\frac{(-K+mN)\widehat{\alpha}}{2K}+u} $$
for $u\in L^{\ast}\cap \alpha^{\perp}$.
By Lemma \ref{T101},  the weight one subspace of $T^1_1$ 
is contained in  
$$\bigoplus_{m=0}^{2K-1} U(\frac{(K+mN)\alpha}{2K})
\otimes e^{\frac{(-K+mN)\widehat{\alpha}}{2K}}$$
with  $-K+mN=0$. Therefore, except for the part in $\widehat{\CH}^{\perp}$, there are no roots in $T^1_1$. Note also that $-K+mN\neq 0$ unless  $N_0=1$. 
\end{proof}

This completes the proof of Proposition \ref{twist}. 

\begin{remark}\label{ginverse}
By exactly the same argument as in Lemma \ref{T101},  we can also show that conformal weight of $L(\lambda_1,...,\lambda_t)^{(\alpha)}$ is 
$\geq 1$ and is $1$ if and only if $\lambda_i=k_i\rho_i/h_i^{\vee}$ for all $i$.  For the $g^{-1}$-twisted module $T^{-1}(=T^{N-1})$, we also have $T^{N-1}_1 \subseteq \widehat{\CH}^{\perp}$. 
\end{remark}

\subsection{$h$-twisted module}\label{sec:6.2}
Set $h=\exp(2\pi i\frac{N\alpha(0)}{K})$.  Then $|h|=K$.  Since the set 
of eigenvalues of $g=\exp(2\pi i\alpha(0))$ on $V$ is $\frac{\bZ}{N}$, the set of eigenvalues of $h$ is $\frac{\bZ}{K}$. 
We note $\langle \frac{N\alpha}{K},\frac{N\alpha}{K}\rangle=\frac{2N_0}{K_0}$ 
and so $R=K/K_0$.  By Lemma \ref{Dual}, $\tilde{V}(h)\cong \tilde{V}$.

\begin{lmm}\label{hreg}
We have  $V^{<h>}_1\subseteq \CH$.  
\end{lmm}

\pr 
By Lemma \ref{dimV1}, $N/K= h_j^\vee/(h^\vee_j+k_j)$ for any $j$. Therefore, 
$(N/K)\alpha = \sum_j \rho_j/ (h^\vee_j+k_j)$. By the same argument as in Lemma \ref{inn-1}, 
\[
0 < |\langle \rho_j, \beta\rangle| < h^\vee_j -1 \quad \text{ for any root $\beta$ of } \CG_j . 
\]  
Therefore, the only integer eigenvalue for $\frac{N}{K}\alpha(0)$ on $V_1$ is $0$ and no root vector is fixed by $h$.
Therefore,  $V^{<h>}_1\subseteq \CH$ as desired. 
\prend

Let $S^j$ be a $h^j$-twisted $V$-module for $j=0, \dots, K-1$. 

\begin{thm}\label{Htwist} 
We have  $S^1_1 \subset \tilde{\CH}$. 
\end{thm}

\pr 
As we explained in Theorem \ref{tVd}, the submodule of $V^{[\frac{N}K \alpha]}$ from the $h$-twisted module $S^1$ is given by 
$$\oplus_{k\in \bZ} X^{(N+kK)}\otimes 
e^{\frac{(-N+kK)\alpha}{2K}}\otimes M(\alpha).$$ 

When $k=0$, the character of $X^{(N)}\otimes 
e^{\frac{-N\alpha}{2K}}$ is equal to that of 
$X^{(N)}\otimes 
e^{\frac{N\alpha}{2K}}$ and $e^{\frac{(0K+N)\alpha}{2K}}=e^{\frac{(-0K+N)\alpha}{2K}}$. Thus,  we have 
$X^{(0K+N)}\otimes 
e^{\frac{(-0K+N)\alpha}{K}}  \subseteq V^{<g>}$, which 
does not contain any  weight one elements. 
For $k=1$, the  weight of $e^{\frac{(-N+K)\alpha}{2K}}$ is equal to 
that of $e^{\frac{(N-K)\alpha}{2K}}$, and the character of 
$ X^{(N+K)}\otimes 
e^{\frac{(-N+K)\alpha}{2K}}$ is equal to that of 
$X^{(N+K)}\otimes 
e^{\frac{(N-K)\alpha}{2K}}$, which is a subspace of $g$-twisted module
and all weight one elements centralize $\CH$ by Lemma \ref{T101}. 
Similarly, for $k=-1$, 
the character of 
$X^{(N-K)}\otimes 
e^{\frac{(-N-K)\alpha}{2K}}$ is equal to that of 
$X^{(N-K)}\otimes 
e^{\frac{(N+K)\alpha}{2K}}$, which is a subspace of $g^{-1}$-twisted module.  
All weight one elements centralize $\CH$, also. 
The remaining cases are 
$k\geq 2$ or $k\leq -2$. 
We first assume $k\geq 2$, then since $K>N$, we have 
\[
(-N+kK)^2- (N+(k-2)K)^2 = 4(k-1)K(K-N) >0,
\]
which implies that the lowest conformal weight of 
$X^{(N+kK)}\otimes 
e^{\frac{(-N+kK)\alpha}{2K}}$ is strictly greater than 
that of 
$X^{(N+kK)}\otimes 
e^{\frac{(N+(k-2)K)\alpha}{2K}}$. 
On the other hand, since $\frac{(N+(k-2)K)\alpha}{2K} = \frac{(N+kK)\alpha}{2K} -\alpha$ and   
$X^{(s)}\otimes e^{\frac{s\alpha}{2K}-\alpha}$ is 
a subspace of $T^1$ for any $s\in \bZ$ by 
\eqref{twistmodule} in Remark \ref{lattice_twisted}, 
$X^{(N+kK)}\otimes e^{\frac{(N+(k-2)K)\alpha}{2K}}$ is a subspace of $g$-twisted module, which does not contain elements whose weights are less than one by Lemma \ref{T101}. So we have a contradiction.  
Similarly, for $k=-n\leq -2$,  
the lowest conformal weight of 
$\oplus_{k\in \bZ} X^{(N-nK)}\otimes 
e^{\frac{(-N-nK)\alpha}{2K}}$ is greater than 
that of $g^{-1}$-twisted module 
$\oplus_{k\in \bZ} X^{(N-nK)}\otimes 
e^{\frac{(N-(n-2)K)\alpha}{2K}}$ whose conformal weight 
is greater than or equal to one and we have a contradiction. 
\prend

By the same argument as in Theorem  \ref{Thm:5.6} and the discussion at the end of Section \ref{sec:4}, we also have the following result. 
 \begin{thm} 
 Let $S^1$ be a $h$-twisted $V$-module. Then 
 $S^1_1\subseteq \Comm(M(\CH),V^{[h]})_1\subseteq \tilde{\CH}$. 
 \end{thm}

As we have shown, $V$ is realizable over $\bQ$ and so the character of 
$T^m$ depends only on $(m,N)$. 

\begin{thm}\label{Noweightzero}
If $(m,N)=1$, then $T^m_1\subseteq \tilde{\CH}$. 
Similarly, if $(n,K)=1$, then $S^n_1\subseteq \tilde{\CH}$. 
\end{thm}

\section{Eigenvalues of $\tilde{g}$ on $\tilde{\CH}$}

In this section, we will prove the following theorem, which is crucial for our arguments.

\begin{thm}\label{KeyThmD}
Let $V\in \CF^2$ and suppose $\rank( V_1) \lneq 24$. Let 
$\alpha$ be  a W-element and $\tilde{g}\in \Aut(V^{[\alpha]})$ the reverse automorphism of $g=\exp(2\pi i\alpha(0))$. Then  
$\tilde{g}$ has an eigenvalue on $\tilde{\CH}$ with order $R$. 
\end{thm} 

As we have shown (see Theorem \ref{Lalpha} and Proposition \ref{N0a}),
 $\Comm(M(\CH),V^{[\alpha]})=
\oplus_{m=0}^{R-1}U(mN _0\alpha)$ and $\CH^{\perp}=\oplus_{m=1}^{R-1} U(mN_0\alpha)_1$ and 
$\tilde{g}$ acts on $U(mN_0\alpha)$ as a multiple by a scalar $e^{2\pi im/R}$. 
Therefore, the above theorem holds if there is an $m$ with $(m,R)=1$ such that $U(mN_0\alpha)_1\not=0$.  Namely, the assertion depends only 
on the structure of $\oplus_{m=0}^{R-1}U(mN_0\alpha)$, but not on construction of $V^{[\alpha]}$. Therefore, we will use an easier orbifold construction $V^{[N_0\alpha]}$ 
associated with $g^{N_0}$, because $\Comm(M(\CH), V^{[N_0\alpha]})$ is also isomorphic to $\oplus_{m=0}^{R-1}U(mN_0\alpha)$ as a VOA and $N_0\alpha\in L^{\ast}$.

Let $g'$ be a reverse automorphism of $g^{N_0}$ in $\Aut(V^{[N_0\alpha]})$.  
Since $R\not=1$, 
\[
\tilde{U}:=\Comm(M(\CH),V^{[N_0\alpha]})=\oplus_{m=0}^{R-1}U(mN_0\alpha)
\] 
is  larger than $U(0)=\Comm(M(\CH),V)$ and so we have $V^{[N_0\alpha]}\in \CF^1$. 
As we have shown, $\tilde{\CH}=\oplus_{m=1}^{R-1} U(mN_0\alpha)_1+\CH$ is 
the unique Cartan subalgebra of $V^{[N_0\alpha]}_1$ containing $\CH$. 
Since $N_0\alpha\in L^{\ast}$, $V_L\subseteq V^{<g^{N_0}>}$ and so 
$\tilde{L}=L+\bZ N_0\alpha$ and we have 
$\Comm(\tilde{U},V^{[N_0\alpha]})=V_{\tilde{L}}$. 

Viewing $V^{[N_0\alpha]}$ as a $\tilde{U}\otimes V_{\tilde{L}}$-module, 
we have $V^{[N_0\alpha]}=
\oplus_{\delta\in (\tilde{L}^{\ast})/\tilde{L}}\tilde{U}(\delta)\otimes 
V_{\tilde{L}+\delta}$. 
We note $\tilde{U}(\mu)=\oplus_{m=0}^{R-1}U(\mu+mN_0\alpha)$.

Since $V^{[N_0\alpha]}\in \CF^1$, 
there is a $g'$-invariant $W$-element $\beta$ such that $(V^{[N_0\alpha]})^{[\beta]}\cong V_{\Lambda}$ by Lemma \ref{KeyLemma3}. 
We note $\beta\in \tilde{\CH}^{<\tilde{g}>}=\CH$. 
Let $\tilde{K}_0$ and $\tilde{N}_0$ be positive integers such that $(\tilde{K}_0, \tilde{N}_0)=1$ and $\langle \beta,\beta\rangle=\frac{2\tilde{K}_0}{\tilde{N}_0}$. 
Let $\tilde{R}$ be the smallest integer such that 
$\tilde{R}\tilde{N}_0 \beta \in \tilde{L}$.
Then as we showed in Proposition \ref{N0a}, $\tilde{N}_0\beta\in (\tilde{L})^{\ast}$ and so 
$\langle N_0\alpha, \tilde{N}_0\beta\rangle\in \bZ$. Furthermore, 
since $\rank((V^{[N_0\alpha]})^{[\tilde{N}_0\beta]})=\rank((V^{[N_0\alpha]})^{[\beta]})=24$, there is a Niemeier lattice 
$E$ such that $V^{[N_0\alpha, \tilde{N}_0\beta]}=V_E$.

Set $W=\Comm(M(\CH), V^{[N_0\alpha,\tilde{N}_0\beta]})$. 
Then 
$$W=\oplus_{n=0}^{\tilde{R}-1}\tilde{U}(n\tilde{N}_0\beta)
=\oplus_{m=0}^{R-1}\oplus_{n=0}^{\tilde{R}-1}U(mN_0\alpha+n\tilde{N}_0\beta).$$

Set $B=L+\bZ N_0\alpha+\bZ \tilde{N}_0\beta$, which is an even lattice and 
set $B_0=\{\mu\in B \mid U(\mu)_1\not=0\}$. We note $|B/L|=R\tilde{R}$. 

We can write $W=\oplus_{\mu\in B/L}U(\mu)$. Note that $U(\mu_1)\times U(\mu_2)\subseteq U(\mu_1+\mu_2)$. 
The statement in Theorem \ref{KeyThmD} is equivalent to $mN_0\alpha\in B_0$ for some $m$ with $(m,R)=1$. 

Since we take $\beta\in \CH$, $B\subseteq \CH$. 
Let $\tilde{f}\in \Aut(V^{[N_0\alpha,\beta]})$ be the reverse automorphism of $f=\exp(2\pi i \beta(0)) \in \Aut(V^{[N_0\alpha]})$.   
We note that $\tilde{f}$ acts on $$\tilde{U}(n\tilde{N}_0\beta)=\oplus_{m=0}^{R-1}U(mN_0\alpha+n\tilde{N}_0\beta)$$ as a 
multiple by a scalar $e^{2\pi in/\tilde{R}}$. 
\medskip

We will first show that $B_0+L$ generate $B$. 

\begin{prop} \label{KPropA}
$W={\rm VOA}(\{U(\mu)\mid \mu\in B_0+L\})$. Equivalently, $B=<B_0+L>$. 
\end{prop}

Since $B/L$ is a finite abelian group of at most rank $2$, there are $\alpha', \beta'\in B$ satisfying 
$$ B/L=(\bZ \alpha'+L)/L\oplus (\bZ\beta' +L)/L. \leqno{\mbox{\emph{Condition D:}}}$$
We first treat pairs $(\alpha',\beta')$ satisfying Condition D and set $R_1=|(\bZ \alpha'+L)/L|$ and $R_2=|(\bZ\alpha'+\bZ\beta'+L)/(\bZ\alpha'+L)|$.  
In this case, $(\bZ \alpha'+L)\cap (\bZ \beta' +L)=L$, which means that 
$n\beta'\in \bZ\alpha'+L$ if and only if $n\beta'\in L$ and so $R_2=|(\bZ\beta'+L)/L|$ and 
the orders of $\exp(2\pi i\beta'(0))$ on $V$ and on $V^{[\alpha']}$ are the same $R_2$.  
Then we can  construct VOAs $V^{[\alpha']}$ and $V^{[\alpha',\beta']}$. 
We note $V^{[\alpha',\beta']}\cong V^{[N_0\alpha,\tilde{N}_0\beta]}\cong V_E$ by 
Proposition \ref{abLindep}.

For $V^{[\alpha']}$, we use  ${\ast}'$ to denote the corresponding subVOAs and sublattices, e.g. 
$\tilde{U}'=\Comm(M(\CH),V^{[\alpha']})=\oplus_{m=0}^{R_1} U(m\alpha')$, 
$\tilde{\CH}'=\oplus_{m=0}^{R_1}U(m\alpha')_1+\CH$ 
is the unique Cartan subalgebra of $V^{[\alpha']}_1$ 
and $\Comm(\tilde{U}',V^{[\alpha']})=V_{\tilde{L}'}$. 
Then $$\Comm(M(\CH),V^{[\alpha',\beta']})
=\oplus_{m=0}^{R_1-1}\oplus_{n=0}^{R_2-1} U(m\alpha'+n\beta')=W.$$ 

\begin{ntn}
Let $\tilde{f'}\in \Aut(V^{[\alpha',\beta']})$ be the reverse automorphism 
of  $f'=\exp(2\pi i\beta'(0))\in \Aut(V^{[\alpha']})$. 
Then there is $\sigma'\in O(E)$ and $\delta\in \bC E$ such that 
$\tilde{f'}=\hat{\sigma}'\exp(\delta(0))$, where $\hat{\sigma}'$ denotes a standard 
lift of $\sigma'$. Since $\sigma'$ acts on 
$E_{\sigma'}$ fixed point freely, we may assume $\delta\in E^{<\sigma'>}$, 
where $E_{\sigma'}=E\cap (E^{<\sigma'>})^{\perp}$. We note $\bC E^{<\sigma'>}=\tilde{\CH}'$.
\end{ntn}

Since $U(m\alpha')\subseteq \Comm(M(\CH),V^{[\alpha']})\subseteq (V^{[\alpha']})^{<f'>}$, 
$\tilde{f}'$ acts on $U(m\alpha')$ trivially.  
Also, $U(m\alpha'+n\beta')\otimes e^{n\beta'}\subseteq V^{[\alpha']}$ and 
$U(m\alpha'+n\beta')=U(m\alpha'+n\beta')
\otimes e^{\frac{(nK'-(nK'_0)R_2)\beta'}{K'}}\subseteq V^{[\alpha',\beta']}$ by the construction. Therefore, $\tilde{f}'$ acts on $U(m\alpha'+n\beta')$ as a multiple by a scalar $e^{2\pi in/R_2}$, 
where $ \langle \beta',\beta'\rangle=2K'_0\in 2\bZ$ and $K'=R_2K'_0$ and $N'=R_2$. 
In particular, $|\tilde{f}'|_{W}|=R_2$. 


\begin{lmm}\label{order2}  For a pair $\alpha',\beta'\in B$ satisfying Condition D, 
there is $u\in <B_0+L>$ such that $u=m\alpha'+n\beta'$ with $(n,R_2)=1$.  
\end{lmm}

\pr Recall that $\bC E\subseteq \oplus U(m\alpha'+n\beta')_1+\CH$ and thus $|\sigma'|=|\tilde{f}'_{|\bC E}|\leq R_2$.  Since $V^{[\alpha',\beta']}=V_E$, we have $|\sigma'|=R_2$ by Lemma \ref{order}, which 
means $$ {\rm Span}_{\bZ}\{n\mid \bC E\cap U(m\alpha'+n\beta')\not=0\}+\bZ R_2=\bZ$$ and so there is $u=m\alpha'+n\beta'\in <B_0+L>$ with 
$(n,R_2)=1$.
\prend

We now start the proof of Proposition \ref{KPropA}.  Let $(s,t)$ be elementary divisors of $B/L$, that is, $B/L\cong \bZ/s\bZ\oplus \bZ/t\bZ$ with $s|t$. 

We first choose $\alpha',\beta'\in B$ satisfying Condition (D) such that 
$|(\beta'+L)/L|=t$, that is, $R_2=t$.  
By Lemma \ref{order2}, there is $u=m\alpha'+n\beta'\in <B_0+L>$ such that 
$(n,R_2)=(n,t)=1$, which means that $|u|=t$ and so 
$<(u+L)/L>$ has a complement in $B/L$.

Without loss, we may choose $\alpha', \beta'\in B$ such that $\alpha'=u$ and $<\beta'+L>$ 
is its complement. In this case, $R_1=t$ and $R_2=s$. 
Again by Lemma \ref{order2}, there is 
$w=m\alpha'+n\beta'\in <B_0+L>$ such that $(n,R_2)=1$. 
Since $w, \alpha'\in <B_0+L>$, we have $n\beta'\in <B_0+L>$. 
Furthermore, since $(n,R_2)=(n,|(\beta'+L)/L|)=1$, $\beta'$ is also 
in $<B_0+L>$. Therefore, $$B=L+<\alpha',\beta'>\subseteq <B_0+L>.$$ 
This completes the proof of Proposition \ref{KPropA}.

\subsection{\bf The proof of Theorem \ref{KeyThmD} }

Next we come back to the proof of Theorem \ref{KeyThmD}. 

\begin{ntn}
Recall that $g'\in \Aut(V^{[N_0\alpha]})$ is 
a reverse automorphism of $\exp(2\pi iN_0\alpha(0))$. 
Since $V_{L+\bZ N_0\alpha}\subseteq (V^{[N_0\alpha]})^{<g'>}$, 
we can induce $g'$ to 
$\tilde{g}'\in \Aut(V^{[N_0\alpha,\beta]})$. 
Let $\tilde{f}\in \Aut(V^{[N_0\alpha,\beta]})$ be a reverse automorphism of $f=\exp(2\pi i\beta(0))\in \Aut(V^{[N_0\alpha]})$. 
Since $\tilde{g}', \tilde{f}\in \Aut(V_{\Lambda})$, there are 
$\tau,\sigma\in Co_0=O(\Lambda)$ and $\delta_1,\delta_2\in \bC \Lambda$ 
such that $\tilde{g}'=\hat{\tau}\exp(\delta_1(0))$ and $\tilde{f}=\hat{\sigma}\exp(
\delta_2(0))$.  
\end{ntn}

\begin{lmm}  $G=<\tau,\sigma>$ is abelian. 
\end{lmm}

\pr 
As we have explained in Theorem \ref{autocommute}, $[\tilde{g}', \tilde{f}]=1$. 
Since $\tau=\tilde{g}'_{|\bC\Lambda}$ and $\sigma=\tilde{f}_{|\bC\Lambda}$,
we have the desired result. 
\prend

Recall that $W=\Comm(M(\CH),V_E)\cong \Comm(M(\CH), V_{\Lambda})$. 
From the process of orbifold constructions,  we have:

\begin{lmm}
$(\bC \Lambda)^G=\CH$ and $W=V_{\Lambda_G}$, where $\Lambda^G=\{v\in \Lambda\mid g(v)=v \text{ for all }g\in G\}$ and $\Lambda_G=\Lambda\cap (\Lambda^G)^{\perp}$. 
\end{lmm} 

\begin{remark}
Since $G$ is cyclic or of rank $2$, we know 
$\dim \CH=4,6,8,10,12,16,24$ from the list of such subgroups in $Co_0$  
(see  Appendix A). 
\end{remark}

\noindent

\begin{prop} \label{KeyPropB} 
Let $V\in \CF^2$ and $\widetilde{V}\not=V_{\Lambda}$. Assume $\CH^{\perp}\cap T^{qN_0}\not=0$ for some $q$. Then there is $j$ such that $T^{j+sqN_0}_1 \not\subseteq \widetilde{\CH}$ for all $s\in \bZ$. Similarly, for $h=\exp(2\pi i \frac{N_0\alpha(0)}{K_0})$ and $h^j$-twisted modules $S^j$, if $\CH^{\perp}\cap S^{qK_0,0}\not=0$, then there is $k$ such that $S^{k+sqK_0,0}_1\not\subseteq \widetilde{\CH}$ for all $s\in \bZ$. In particular, since $T^0_1\subseteq \CH$ and $T^j_1\subseteq \CH$ for $(j,N)=1$, we have $N_0\!\not\!|j$ and $(j+sqN_0, N)\not=1$. 
\end{prop}

\pr  By the assumption, there is $0\not=\beta\in U(qN_0\alpha)_1$. Therefore, there is a root vector  $e^{\gamma}\in \widetilde{V}_1$ such that $\beta(0)e^{\gamma}=\lambda e^{\gamma}$ with $\lambda\not=0$. Express $e^{\gamma}=\oplus e^{\gamma}_i \in \oplus_j T^{i,0}_1$ and $e^{\gamma}_j\not=0$ for some $j$. In this case, since $\beta(0)e^{\gamma}=\lambda e^{\gamma}$, $e^{\gamma}_{j+qN_0}\not=0$. Namely, $T^{j+sqN_0}_1 \not\subseteq \widetilde{\CH}$ for any $s\in \bN$ as we desired. For $h$-twisted module, we can prove it similarly. \\

\noindent
{\bf Lemma}\quad If $\sigma=1$, i.e. $\rank(V^{[N_0\alpha]}_1)=24$, then Theorem \ref{KeyThmD} holds. \\

\pr  If $\sigma=1$, then $\widetilde{\CH}=\bC\Lambda$. Since 
$\dim (\bC\Lambda)^{<\tau>} \geq 4$, the order $|\tau|\leq 15$ (see Theorem \ref{thm:A1}). 
In this case, the list of elements in $Co_0$ shows that $m_{|\tau|} >0$ in the frame shape of $\tau$ and so $\tau$ has an eigenvalue of order 
$|\tau|$ on $\bC\Lambda$, which proves Theorem \ref{KeyThmD}. \prend

So, from now on, we will treat the case where $\sigma\not=1$ and 
$\rank(V^{[N_0\alpha]}_1)<24$.

\subsubsection{The case $\dim \CH=4$.}  \label{s:6.11}
We first prepare some information about $R$ for the case $\rank(V_1)=4$.
Since $\dim \CH=4$ and $\dim V_1>24$, $V_1$ is one of $B_4,C_4,D_4,F_4,G_2^{\oplus 2}$. 
Since its dimensions are $36, 36, 28, 52, 28$, respectively, 
we have $\langle \alpha,\alpha\rangle=\frac{2\dim V_1}{\dim V_1-24}=
6,6,14,26/7,14$ and so $N_0=1,1,1,7,1$ and $N$ is a multiple of $14,10,6,18,12$ and $R$ is a multiple of $14,10,6,18/7,12$, respectively. 
On the other hand, since $\dim(\bC\Lambda^{\tau})\geq 4$, 
$R=|\tau|\leq 15$ by Theorem A.1.  Hence $V_1\not\cong F_4$ and 
$R\in \{6,10,12,14\}$ if $V\in \CF^2$ with $\rank V_1=4$.

\subsubsection{Cyclic case }
Assume that $G$ is cyclic, say $G=<\xi>$ with $\xi\in Co_0$. 
Then  $\dim\bC\Lambda^{\xi}\geq 4$ and hence $|\xi|\leq 15$. 
Clearly, we may assume $\tilde{R}=|\sigma|\not=1$ and 
$|G/\langle\sigma\rangle|=R\not=1$ and so $|\xi|$ is a composite number, that is, 
$4,6,8,9,10,12,14,15$ and we may assume $\sigma=\xi^R$. 
For general cases,  
if we prove that $\xi$ has an eigenvalue of order $R$ 
on $\bC \Lambda$, 
then its eigenvectors have to be in $\bC\Lambda^{<\xi^R>}=\tilde{\CH}$ since $\sigma=\xi^R$ and $\tilde{\CH}=(\bC\Lambda)^{<\xi^R>}$.  

Since $\sigma\not=-2A,3C,5C$, the above statement follows from the following statement: 
if $\xi\in Co_0$ has an order of a composite number $n$ and 
$\dim \bC\Lambda^{<\xi>}\geq 4$ and has no eigenvalue on $\bC\Lambda$ of order $m$ for 
$m|n$, then $\xi^m$ belongs to  $-2A$, $3C$, or $5C$.

We will check all elements in $Co_0$. See the data for subgroups of Conway group 
$Co_0$ in Appendix A. Since $\dim \bC\Lambda^{<\xi>}\geq 4$, it is enough to check $\xi$ with $|\xi|\leq 15$. 
First we will explain the notion of frame shape.  An element $\xi\in Co_0$ has the frame shape  $\prod_{n} n^{m_n}$ means that the characteristic polynomial of $\xi$ on $\bC \Lambda$ is given by  $\prod_{n} (x^n-1)^{m_n}$. 
Therefore, if $\xi$ has a frame shape $\prod_n n^{m_n}$, then 
the multiplicity of the eigenvalue $e^{2\pi i/R}$ of $\xi$ on $\bC\Lambda$ 
is given by $\sum_{h}{m_{Rh}}$. 
Set $k=|\xi|$. It is easy to check that  $m_{k} >0$ for every element $\xi\in Co_0$ (except for $30D$) which has non-trivial fixed points. If $\xi$ has positive frame shape, i.e., none of $m_n$ is negative, then $\xi$ has an eigenvalue of any order that divides $|\xi|$.
By the same reason, if the denominator of the frame shape is just $1^k$, then 
it also has eigenvalue of order $R$. 

Therefore, we assume $m_n<0$ for some $n>1$. We first treat the cases $\dim(\bC\Lambda)^{<\xi>}>4$.  There are only $3$ conjugacy classes of $Co_0$ which satisfy these conditions. They are  
$-4A,6C$, and $-6D$ (see Theorem A.1 in \S A). Next we compute the values of  
$\sum_h m_{Rh}$  and $\xi^R$ with  $m_R<0$ and $R>1$. The results are follows:
$$\begin{array}{rlllll}
\xi&\mbox{frame shape}&\dim \bC\Lambda^{<\xi>}&R \mbox{ with }m_{R}<0& 
\sum_h m_{hR} &\xi^{R}   \cr
\hline 
  -4A&1^84^8/2^8                &8        &2&8-8=0&(-4A)^2=-2A \\
   6C&1^42^16^5/3^4            &6        &3& 5-4=1& \\
 -6D&1^53^16^4/2^4             & 6       &2&4-4=0&(-6D)^2=3C \\
\end{array}$$
So, $\xi$ has also an eigenvalue of order $R$ or $\xi^R\not=\sigma$. 

We next treat the case $\bC\Lambda^{<\xi>}=4$. 
In this case, $R=6,10,12,14$ as we showed in \S 6.1.1. 
It is enough to check 
only $\xi\in Co_0$ which has $R>1$ in the denominator of its frame shape such that 
$\sum_h m_{hR}=0$. By Theorem \ref{thm:A1}, there is only one such case: $\xi=-12E$, which has a frame shape $1^23^24^212^2/2^26^2$ and $R=6$. However, $(-12E)^6=-2A$, which contradicts the fact $\sigma\not= -2A$. 
Therefore, we have proved that if $\xi^R=\sigma$, then $\xi$ has also an eigenvalue of 
order $R$ even if $\dim \CH=4$. 

As a consequence, Theorem \ref{KeyThmD} holds if $G$ is cyclic.

\subsubsection{Non cyclic case}
\noindent \textbf{Part 1: The case $\dim \CH=4$.}  
\medskip

We first note  $G\cong \bZ_2\times \bZ_2, \bZ_2\times \bZ_4, 
\bZ_3\times \bZ_3, \bZ_2\times\bZ_6, \bZ_4\times \bZ_4, \bZ_2\times \bZ_8, \bZ_3\times \bZ_6, \bZ_2\times \bZ_{12}, \bZ_5\times\bZ_5, \bZ_3\times \bZ_9$ by Theorem \ref{Conway}.  As we explained in \S \ref{s:6.11},  
$V_1$ is isomorphic to $B_4,C_4,D_4$,or $G_2^{\oplus 2}$. 
Since $G$ is not cyclic, 
$G$ is one of the above 10 groups in Theorem \ref{Conway}, 
that is, $|G|=4,8,9,12,16,18,24,25,27$. Since $R$ divides $|G|$, 
$B_{4}$ with $R=14$ and $C_{4}$ with $R=10$ don't appear in this case. 
So the possible cases are $D_{4}$ and $G_{2}+G_{2}$. \\

\noindent
{\bf Case $V_1\cong D_{4}$.} \quad  In this case, $V_1\cong D_{4,36}$. Then $N_0=1$, $|\tau|=|g|=6$ and $G\cong \bZ_2\times \bZ_6$ or $\bZ_3\times \bZ_6$. 
In either case, $\bC\Lambda^{\sigma}=\tilde{\CH}> \CH$ and so 
there is  $ 0<m,n<6$ such that $\delta\in (T^m)_1\cap \CH^{\perp}$ and 
$v\in (T^n)_1-\CH^{\perp}$ such that $\delta(0)v\not=0$.  
If $(m,6)=1$, then $\tau$ has an eigenvalue of order $6$ on $\CH^{\perp}$ 
and so we may assume $(m,6)\not=1$, that is, 
$T^1\cap \tilde{\CH}=0$ and so we also have $(n,6)\not=1$. 
If $m=2$ or $4$, then one of $n, n+2$, or $n+4$ is congruent to $0$ or $\pm 1 \mod 6$. 
However, for such a number $t$, $T^{t}_1-\tilde{\CH}=0$ by Proposition \ref{KeyPropB}, which is a contradiction. 
If $m=3$, then one of $n, n+3 \equiv 0, \pm 1 \mod 6$ and we also have a contradiction by Proposition \ref{KeyPropB}.  

\medskip

\noindent
{\bf Case $V_1\cong G_{2}^{\oplus 2}$.} \quad 
In this case, $V_1\cong G_{2,24}\oplus G_{2,24}$. 
Since $|\tau|=12$ and the rank of $\Lambda^{\tau}\leq 4$ for any element of order $12$, 
$\mathrm{rank}\,(V^{[\beta]}_1)=4$ and $\bC\Lambda^{<\tau>}=\CH$. 
The possible choice for $\tau$ is 12I or 12J; otherwise $<\tau>$ contains $-2A$ or $3C$.  
From Harada-Lang's list, the Gram matrix of $\Lambda^{<\tau>}$ is either 
\[
\begin{pmatrix}
6&2&2&2\\
2&4&0&0\\
2&0&4&0\\
2&0&0&4
\end{pmatrix}
\quad \text{ or }\quad 
\begin{pmatrix}
4&2&0&0\\
2&4&0&0\\
0&0&8&4\\
0&0&4&8 
\end{pmatrix}
\]
and their determinant is 
$2^6\times 3$ or $2^6\times 3^2$.  On the other hand, since $L=\sqrt{24}(A_2+A_2)$ and $\tilde{L}=L+\bZ\beta$, 
the discriminant of $\tilde{L}$ is $(24)^4\times 3^2/2^2$, which is none of the above. 

\medskip
\noindent \textbf{Part 2. The case dim $\CH>4$.}  
By Theorem \ref{Conway}, $G$ is one of 
$$\begin{array}{llll}
&\dim \CH &   &\cr
(1)&6&\bZ_2\times \bZ_2=\{1A,2A,-2A, 2C \}& \cr
(2)&6&\bZ_2\times \bZ_2=\{1A,2C, 2C,2C \}&\cr
(3)&6&\bZ_2\times \bZ_4=\{1A,2A,2A,-2A, 4C, 4C, 4C, 4C \}& \cr
(4)&6&\bZ_2\times \bZ_4=\{1A,2A,2C,2C,4C,4C,4C,4C\}& \cr
(5)&6&\bZ_2\times \bZ_4=\{1A,2A,2A,2C,4C,4C,4D,4D \}& \cr
(6)&6&\bZ_2\times \bZ_4=\{1A,2A,2A,2A,4C,4C,-4C,-4C \}& \cr
(7)&6&\bZ_2\times \bZ_4=\{1A,2A,2A,2A,4D,4D,4D,4D \}&\cr
(8)&6&\bZ_3\times \bZ_3=\{1A,3B,3B,3B,3B,3B,3B,3C,3C\}&\cr
(9)&6&\bZ_2\times \bZ_6=\{1A,2A,2A,2A,3B,3B,\mbox{6 elts. of }6E\}&\cr 
(10)&6&\bZ_4\times \bZ_4=\{1A,2A,2A,2A,\mbox{12 elts. of }4C \}&\cr
(11)&8&\bZ_2\times \bZ_2=\{1A,2A,2C,2C\}&\cr
(12)&8&\bZ_2\times \bZ_2=\{1A,2A,2A,-2A\}&\cr
(13)&8&\bZ_2\times \bZ_4=\{1A,2A,2A,2A,4C,4C,4C,4C\}&\cr
(14)&8&\bZ_2\times \bZ_4=\{1A,2A,2A,-2A,\mbox{4 elts. of }-4A \}& \cr
(15)&8&\bZ_3\times \bZ_3=\{1A, 3B, 3B, 3B, 3B, 3B, 3B, 3B, 3B \}&,\cr
(16)&10&\bZ_2\times \bZ_2=\{1A,2A,2A,2C\}&\cr
(17)&12&\bZ_2\times \bZ_2=\{1A,2A,2A,2A\}&\cr
\end{array}$$

\noindent
\subsubsection{Character}
In order to check that $\tau$ has an eigenvalue of order $R$ on 
$\bC\Lambda$, we will use the character theory. 
Let $\theta$ be a linear character of $G$ defined by 
$\theta(\tau^m\sigma^n)=\exp(2\pi im/R)$. Then $\ker \theta=< \sigma> $ and $\theta$ is  a faithful 
character of $G/\langle \sigma\rangle$ of order $R$. 
In order to prove that $\tau$ has an eigenvalue of order $R$ on 
$\tilde{\CH}=(\bC\Lambda)^{<\sigma>}$, 
it is enough to show that $G$ has a character $\theta$ on $\bC\Lambda$. Equivalently, it is enough to show 
$0\not=\langle \chi,\theta\rangle$, which is calculated by 
$\frac{1}{|G|}\sum_{g\in G}\theta(g)\chi(g^{-1})$, 
where $\chi$ is the irreducible character of $Co_0$ of degree $24$ 
afforded by the action on $\bC\Lambda$. 
We note that if $g\in Co_0$ has a frame shape $\prod_{n=1}^{\infty} n^{m_n}$, 
then $\chi(g)=m_1$.

\subsubsection{Case 1: $R$ is a prime number}
If $R$ is a prime, then $R=2$ or $3$ since $R||G|$. 
If $\sigma$ satisfies the condition 
$$(S): \qquad \bC\Lambda^{\sigma}\not=\bC\Lambda^G=\CH,$$ 
then clearly $\tau$ acts on $\CH^{\perp}$ non-trivially and so 
$\tau$ has an eigenvalue of order $R$. 
We note $\sigma\not=-2A$ nor $3C$.  So we first check the condition $(S)$ 
for the above 17 groups and $<\sigma>$ with prime index $R$:\\
Since $\dim(\bC \Lambda)^{2A}=16$, $\dim(\bC\Lambda)^{2C}=12$, $\dim(\bC\Lambda)^{4C}=10$, $\dim(\bC\Lambda)^{4D}=8$, $\dim(\bC\Lambda)^{3B}=12$, 
(1),(2),(3),(4),(5),(7),(8),(9),(10),(11),(12),(13),(15),(16),(17) satisfy the condition (S) and so 
they don't appear in the case if $R$ is a prime number.
The remaining cases are (6) and (14), that is,  
$$\begin{array}{llll}
(6)&\dim \CH=6&\bZ_2\times \bZ_4=\{1A,2A,2A,2A,4C,4C,-4C,-4C \}& \sigma=-4C  \cr
(14)&\dim \CH=8&\bZ_2\times \bZ_4=\{1A,2A,2A,-2A,\mbox{4 elts of }-4A \}&\sigma=-4A \end{array}
$$
Since $\LCM(\{\tilde{r}_j\tilde{h}_j^{\vee}/\tilde{N}_0\mid j\})=|\sigma|=4$,   $\tilde{r}_j\tilde{h}_j^{\vee}=4\tilde{N}_0$ for some $j$.
Then $2| \tilde{h}_j^\vee$ and there exist roots $\mu^s$, $s=1,2$
with  $\langle \alpha,\mu^s\rangle \cong s/N\, \mathrm{mod}\,\mathbb{Z}$.
For $-4C$ or $-2A=(-4A)^2$, twisted module of $V_\Lambda$ has top
weight $1$. Then
$\tilde{f}$ (resp., $\tilde{f}^2$) -twisted module has no weight one
element if $\sigma=-4C$ (resp., -$4A$).
As the root vector of $\mu^s$ is in $\tilde{f}^s$-twisted module,
$\sigma \neq -4C, -4A$.  


\subsubsection{Case 2: $R$ is a composite number}
Therefore, the possible cases are 
$G=\bZ_2\times \bZ_4, \bZ_4\times \bZ_4, \bZ_2\times \bZ_6$. \\
(i) If $G=\bZ_2\times \bZ_4$, i.e. cases (3),(4),(5),(6),(7),(13),(14), then $|\tau|=4$ and 
$|\sigma|=2$. Then the set of elements of order $4$ is $\{\tau,\tau^{-1}, \tau\sigma, 
\tau^{-1}\sigma\}$ and $\theta(\tau)=\theta(\tau\sigma)=\sqrt{-1}$ and 
$\theta(\tau^3)=\theta(\tau^3\sigma)=-\sqrt{-1}$. 
Since if $g^{-1}$ is conjugate to $g$ in $Co_0$ for all element $g\in Co_0$, 
we have $\chi(\tau)=\chi(\tau^3)$ and $\chi(\tau\sigma)=\chi(\tau^3\sigma)$. 
Hence we have 
$$\begin{array}{rl}
8\langle \theta,\chi\rangle=&i \chi(\tau)-\chi(\tau^2)-i\chi(\tau^3)+\chi(\tau^4)+i\chi(\tau\sigma)-\chi(\tau^2\sigma)-i\chi(\tau^3\sigma)+\chi(\tau^4\sigma)\cr
=&24-\chi(\tau^2)-\chi(\tau^2\sigma)+\chi(\sigma). \end{array}$$ 
Since $|\tau^2|=|\tau^2\sigma|=|\sigma|=2$, $\tau^2,\tau^2\sigma, \sigma\in \{2A,-2A,2C\}$. 
We note $\chi(-2A)=-8, \chi(2A)=8, \chi(2C)=0$. Hence  
$\chi(\tau^2), \chi(\tau^2\sigma)\geq -8$, but 
$\chi(\sigma)\geq 0$ since $\sigma\not=-2A$.
Therefore, 
$8\langle \theta,\chi\rangle=24-\chi(\tau^2)-\chi(\tau^2\sigma)+\chi(\sigma)>0$. \\
(ii) If $A=\bZ_4\times \bZ_4$, i.e. case (10), then all elements in $G$ of order $4$ is $-4C$ and all involutions are $2A$. In this case, we have 
$$\begin{array}{rl}
16\langle \theta, \chi\rangle=&\!\! i \chi(\tau)-\!\chi(\tau^2)-\!i\chi(\tau^3)+\chi(1)+\!i\chi(\tau\sigma)-\!\chi(\tau^2\sigma)-\!i\chi(\tau^3\sigma) +\!\chi(\sigma)+i\chi(\tau\sigma^2)\cr
&-\chi(\tau^2\sigma^2)-\!i\chi(\tau^3\sigma^2)+\chi(\sigma^2)+i\chi(\tau\sigma^3)-\chi(\tau^2\sigma^3)-i\chi(\tau^3\sigma^3) +\chi(\sigma^3) \cr
=&-\chi(\tau^2)+\chi(1)-\chi(\tau^2\sigma)+\chi(\sigma)-\chi(\tau^2\sigma^2)+\chi(\sigma^2)-\chi(\tau^2\sigma^3)+\chi(\sigma^3)\cr
=&-8+24-4+4+8-8-4+4=16\not=0. \end{array}$$
(iii) The last case (9), that is, $G\cong \bZ_2\times \bZ_6$. In this case, 
$|\sigma|=2$, or $3$. However, if $|\sigma|=3$, then $G/<\sigma>$ is not cyclic, which contradicts the choice of $\tau$ and $\sigma$. 
Therefore, $|\sigma|=2$ and so 
$\tau=6E$ and $\sigma=2A$. 
We note $\chi(6E)=2, \chi(3B)=6$ and $\chi(2A)=8$. 
In this case, $\theta(\sigma^i\tau^j)=(-\omega)^j$ where $\omega=e^{2\pi i/3}$.
 
$$\begin{array}{l}
12\langle \theta, \chi\rangle=\!-\omega \chi(\tau)+\!\omega^2\chi(\tau^2)-\!\chi(\tau^3)+\!\omega \chi(\tau^4)-\omega^2\chi(\tau^5)+\chi(1)-\!\omega \chi(\tau\sigma)+\omega^2\chi(\tau^2\sigma)\cr
\mbox{}\qquad \qquad-\chi(\tau^3\sigma)+\omega \chi(\tau^4\sigma)-\omega^2\chi(\tau^5\sigma)+\chi(\sigma) \cr
\mbox{}\quad=
\chi(6E)(-\omega-\omega^2-\omega+\omega^2+\omega-\omega^2)+\!\chi(3B)(\omega^2+\omega) +\!\chi(2A)(-1-1+1)+\!\chi(1)\cr
\mbox{}\quad=2-6-8+24=12\not=0.\end{array}$$
This completes the proof of Theorem \ref{KeyThmD}.




\section{Proof of the main Theorem} 
In this section, we will complete the proof of our main theorem. 
As a corollary of Theorem \ref{KeyThmD} and Proposition \ref{KeyPropB}, we have the following theorem.

\begin{thm}\label{ThmE} 
If $V\in \CF^2$ and $\widetilde{V}\not=V_{\Lambda}$, then there is $j\not\in N_0\bZ$ such that $T^{j+sN_0,0}_1\not=0$ for all $s=0,...,R-1$. Similarly, for $h=\exp(2\pi i \frac{N_0\alpha(0)}{K_0})$ and $h$-twisted parts 
$S^{j,0}$, there is $k\not\in K_0\bZ$ such that $S^{k+sK_0,0}_1\not=0$ for all $s$. 
\end{thm}

We also have:

\begin{lmm}\label{Noroot} 
If $R$ is a prime number, then there is no root of $\tilde{V}_1$ in $\alpha^{\perp}$. 
\end{lmm}

\pr Suppose false and let $\delta\in \alpha^{\perp}$ be a positive root with 
$e^{\delta}\in \tilde{V}_1$.  
Since 
\[
V^{[\alpha]}=\bigoplus_{m=0}^{N-1}\bigoplus_{n\in \bZ}X(\frac{(mK_0+nN_0)\alpha}{2K_0})
\otimes M(\alpha)\otimes e^{\frac{(-mK_0+nN_0)\alpha}{2K_0}}\text{  and } 
\tilde{\CH} \subseteq \bigoplus_{s=0}^{R-1} X(sN_0\alpha)+\bC \alpha,
\]  
there are $m,n$ such that the image $\pi_{m,n}(e^{\delta})$ of the  projection 
$$\pi_{m,n}:V^{[\alpha]}\to \bigoplus_{s=0}^{R-1} X(\frac{(mK_0+nN_0)\alpha}{2K_0}+sN_0\alpha)\otimes M(\alpha)\otimes e^{\frac{(-mK_0+nN_0)\alpha}{2K_0}}$$
given by (3.9) is nonzero.  
Since $\langle \delta,\alpha\rangle=0$, $\alpha(0)e^{\delta}=0$ and so 
we have $-mK_0+nN_0=0$, which means $N_0|m$ and $\frac{mK_0+nN_0}{2K_0}
=m=sN_0$ for some $s\in \bZ$.   
Then $$X(\frac{(mK_0+nN_0)\alpha}{2K_0})\otimes M(\alpha)=X(sN_0\alpha)\otimes M(\alpha)\subseteq T^{sN_0,0}.$$ 
In particular, $T^{sN_0,0}_1\not\subseteq \tilde{\CH}$. Furthermore, 
there are $t\in \bZ$ and $\mu\in U(tN_0\alpha)$ such that 
$0\not=\mu(0)e^{\delta}\in \bC e^{\delta}$. Hence $T^{sN_0+atN_0,0}\not\subseteq 
\tilde{\CH}$ for all $a\in \bZ$. 
Since $R=N/N_0$ is a prime number by the assumption,
 we have $V^{<g>}_1=T^{0,0}_1\not\subseteq \CH$, which contradicts the property of the W-element $\alpha$. 

\prend

\begin{lmm}
Let $p$ be a prime. Let $j$ be a positive integer such that $(j,p)=1$. Then there is a $0 \leq t \leq R$ such that $(j+tp, R)=1$. 
\end{lmm}

\pr  Denote $R=p^kq$ with $(p,q)=1$. Then 
$j+tp\equiv j+t'p \mod q$ if and only if $t-t'\equiv 0 \mod q$. That means the map 
$\phi:\bZ \to \bZ/q\bZ$ with $t\mapsto j+tp$ is surjective. Therefore, there is a $t$ such that $j+tp$ is invertible in $\bZ/q\bZ$. On the other hand,  $(j, p)=1$ implies 
$j+tp$ is invertible in $ \bZ/p^k\bZ$, also. Hence,  $j+tp$ is invertible in 
$\bZ/q\bZ\times \bZ/p^k\bZ \cong \bZ/R\bZ$ as desired. 
\prend

Next, we will prove the following key theorem: 

\begin{thm}\label{KeyThmF}  
If $V\in \CF^2$ and $N_0$ or $K_0$ are $1$ or a prime number, then $\widetilde{V}$ is isomorphic to the Leech lattice VOA.
\end{thm}

\pr
We will prove the first case. The second case is similar.  Suppose $\widetilde{V}_1$ semisimple and $N_0$ is a prime number. Then by the above proposition, there is $j\in \bZ-\bZ N_0$ such that 
$T^{j+tN_0,0}_1\not\subseteq \widetilde{\CH}$ for all $t$. 
If $(j,N)=1$, then we have already shown 
$T^{j,0}_1\subseteq \widetilde{\CH}$, which contradicts the choice of $j$. 
Since $j\not\in \bZ N_0$ and $N_0$ is a prime number, 
we have $(j,R)\not=1$. Then there is $0\leq t\leq R-1$ such that $(j+tN_0,R)=1$ by Lemma \ref{KeyThmF}.  Since $(j,N_0)=1$ and 
$N_0$ is a prime number, $(j+tN_0,N_0R)=1$, which contradicts $
T^{j+tN_0,0}_1\not\subseteq \widetilde{\CH}$. \prend

We continue the proof of the main theorem. 
By Theorem \ref{KeyThmF}, we may assume that $K_0$ and $N_0$ are both composite numbers. 
We may also assume $\rank V_1=16,12,10,8,6,4$.

\subsection{The case where $N_0$ is even} 
Then $K_0$ is odd. Since $(K_0-N_0)|24$, we have  $K_0-N_0=1$ or $K_0-N_0=3$. 
\subsubsection{$K_0-N_0=1$}
In this case, $K_0=N_0+1$ is an odd composite number. Therefore,  $K_0=9,15,21,25,27$ or $K_0\geq 33$ and $\dim V_1=24K_0$. 
The possible simple component  $\CG_j$ with $N_0|h^{\vee}$ and $\rank( \CG_j) \leq 16$ are 
$$\begin{array}{lll}
N_0=9-1=8 &\dim V_1=216& A_7, A_{15}, C_7, C_{15}, D_5, D_9, D_{13} \cr
N_0=15-1=14 &\dim V_1=360& A_{13}, C_{13}, D_8, D_{15} \cr
N_0=21-1=20 & &D_{11} \cr
N_0=25-1=24 & &D_{13} \cr
N_0=27-1=26 & &D_{14} \cr
N_0\geq 33-1=32 & &\emptyset \cr
\end{array}$$
Since $\dim \CH=4,6,8, 10,12,16$ and $V_1$ is a sum of the above components 
in each row, the possible cases are $N_0=8$ and 
$V_1\cong D_5+A_7, D_5+C_7, D_9+A_7, D_9+C_7$, 
or $N_0=14$ and $V_1\cong D_8, D_8+D_8$.  However, 
their dimensions are $108$, $150$, $216$, $258$ and $120$, $240$. 
Therefore, the possible case is $V_1\cong D_9+A_7$ with $K_0=9$ and $N_0=8$.

\subsubsection{$K_0-N_0=3$}
{In this case, $(3,N_0)=1$ and $(3,K_0)=1$. By a direct calculation, it is easy to check that for even composite number $N_0$, $K_0=N_0+3\in \{ 2+3,4+3,8+3,10+3,14+3,16+3,20+3\}$ are prime numbers if $N_0\leq 20$. We hence have $N_0\geq 22$. 
Since $N_0|h^{\vee}_j$ and $N_0$ even and $(N_0,3)=1$, 
the possible component $\CG_j$ with $N_0|h^{\vee}$ and rank in $\{16,12,10,8,6,4\}$ is only $D_{12}$ and so $\dim V_1=12\times 23=276$. 
However, in this case, $\langle\alpha,\alpha\rangle=\frac{2K_0}{N_0}=
\frac{2\times 256}{21}$, which is a contradiction.

\subsection{The case where $N_0$ is odd}
Since $N_0$ is an odd composite number, $3|N_0$ or $N_0\geq 25$. If $N_0\geq 25$, then since 
$N_0|h^{\vee}_j$ and $\rank(V_1)\leq 16$, 
the possible components of $V_1$ are $B_{n}$ with $n\geq 13$. 
Namely, $V_1\cong B_{16}$ and $\dim V_1=528$ and $N_0|31$. 
Hence $\frac{K_0-N_0}{N_0}=\frac{24}{528-24}=\frac{2}{42}$ and so 
we have $21|N_0$, which contradicts $N_0|31$. 
Therefore, we may assume $3|N_0$ and $N_0\leq 24$, which means 
$N_0=9, 15, 21$ and $(K_0,3)=1$. 
 
Since $(K_0-N_0)|24$ and $(K_0-N_0,3)=1$, 
we have $K_0-N_0=1,2,4,8$ and so \\
$$\begin{array}{rl}
(N_0,K_0)=&(9,10),(9,11),(9,13),(9,17),(15,16),(15,17),(15,19),(15,23),\cr 
&(21,22),(21,23),(21,25),(21,29).\end{array}$$ 
Furthermore, since $K_0$ is also a composite number, the possible cases are 
$$\begin{array}{lccccc}
(N_0,K_0)&=&(9,10), &(15,16),  &(21,22), &(21,25)\cr
\dim V_1=\frac{24K_0}{K_0-N_0}&=&240 & 384 & 528& 150 
\end{array}$$

\subsubsection{The case where $N_0=9$}
By Proposition \ref{dimV1}, $N_0|h_i^{\vee}$, i.e. $9|h_i^{\vee}$. If $h_i^{\vee}=9$, then the possible components for $(\CG_i, \dim \CG_i)$ are $(A_8,80), (B_5,55), (C_8,136), (F_4,52)$. If there is $j$ such that $h_j^{\vee}>9$, then $h_j^{\vee}=18,27,36,45,..$. Since 
$\dim \CG_j\leq 240$, $(D_{10},190),(E_7,133)$ are the only possible pairs for 
$(\CG_j, \dim \CG_j)$. Therefore, since the total dimension is 240, 
the possible cases of $V_1$ are 
$V_1\cong C_{8,1}+F_{4,1}^2$,  $E_{7,2}+B_{5,1}+F_{4,1}$.

\subsubsection{The case where $N_0=15$, } 
We note $N_0|h^{\vee}_j$. The possible components $(\CG_j, \dim \CG_j)$ 
are $(A_{14},224), (B_8,136),(C_{14},406)$ with $h^{\vee}=15$ or 
$(E_8,248),(D_{16},496)$ with $h^{\vee}=30$ or $(B_{23},1081)$ 
with $h^{\vee}=45$. Therefore, the possible case of $V_1$ with 
$\dim V_1= 384$ is $V_1\cong E_{8,2}+B_{8,1}$ and $K_0-N_0=1$. 

\subsubsection{The case where $N_0=21$}
The possible components for 
$(\CG_j, \dim \CG_j\leq 528)$ are $(A_{20},440),(B_{11},253)$ with $h^{\vee}=21$, 
but none of their combinations gives  the desired rank $16,12,10,8,6$ or $4$.  

\subsection{Final proof}
There are remaining only four possible cases for $V_1$: 
$$V_1\cong D_{9,2}+A_{7,1}, \quad C_{8,1}+F_{4,1}^2,\quad E_{7,2}+B_{5,1}+F_{4,1},\quad \text{ or } 
\quad E_{8,2}+B_{8,1}.$$ 
In particular, $\rank V_1=16$, $K_0-N_0=1$, $R=2$ and $N_0=8,9,9,15$. 
As we have shown, $R=|\tau|\not=1$ and 
$\tau$ acts on $\CH^{\perp}$ faithfully. That means $\mathrm{rank}(\tilde{V}_1) \gneq \mathrm{rank}(V_1)=16$ and we have $\rank(\tilde{V}_1)=24$. Therefore, 
there is a Niemeier lattice $E$ such that $\tilde{V}\cong V_E$. 
Suppose $E\not=\Lambda$. 
Since $R=2$ is a prime number, by Lemma \ref{Noroot}, 
no root in $\tilde{V}_1$ is orthogonal to $\alpha$, we can define a 
root $\beta\in \tilde{V}_1$ to be positive if $\langle \beta,\alpha\rangle>0$. 
Let $x_1,...,x_{24}$ be a set of simple roots for  $\tilde{V}_1=(V_E)_1$, i.e., 
the set of all indecomposable positive roots.  
We can also express $\alpha=\sum a_ix_i$ with $a_i\in \bR$. 
Since $R=2$,  $\tilde{\CH}=\bC E\subseteq  
T^0\oplus T^{N_0}$. 
Let $e^{x_i}\in \tilde{V}_1$ be a root vector associated with the root $x_i$. 
Then $e^{x_i}\in \cup_{m=1}^{R-1} (T^m+T^{m+N_0})$.  
Recall that $T^m=\oplus_{n} P(m,n)$, where 
$P(m,n):=X(\frac{(mK_0+nN_0)\alpha}{2K_0})\otimes 
e^{\frac{(-mK_0+nN_0)\alpha}{2K_0}}$. 
Note that $P(n,m)$ is also isomorphic to a subspace of 
$\exp(2\pi i \frac{nN_0\alpha(0)}{K_0})$-twisted module $S^n$. 
Therefore, if $P(m,n)_1\not\subseteq \tilde{\CH}$, then $(m,N_0)\not=1$, 
$(n,K_0)\not=1$, $m\not\in N_0\bZ$ and $n\not\in K_0\bZ$.

\subsubsection{Case $N_0=8$, $K_0=9$, $R=2$}\label{caseN0=8}
In this case,  we have $P(m,n)= X(\frac{(9m+8n)\alpha}{18})\otimes 
e^{\frac{(-9m+8n)\alpha}{18}}$.  
For a simple root $x_i$, we write $x_i= k_i\alpha +x_i'$, where $x_i'\in (\bQ \alpha)^\perp$.  If $P(m,n)_1\not\subseteq \tilde{\CH}$, then $(n,9)\not=1$ and $(m,8)\not=1$, that is, $n=3n_0$ and $m=2m_0$ and  
$\frac{(-9m+8n)\alpha}{18}=\frac{(-3m_0+4n_0)\alpha}{3}$. Therefore, $k_i= \frac{(-3m_0+4n_0)}{3}$ for some $m_0, n_0$.  Since $x_i$ are positive roots, $\frac{(-3m_0+4n_0)}{3} >0$ and we have 
$\langle x_i,\alpha\rangle\geq \langle \frac{\alpha}{3},\alpha\rangle
=\frac{18}{24}=
\frac{3}{4}$. 

Let $\{\varpi_1, \dots, \varpi_{24}\}$ be the set of fundamental weights. 
Since the irreducible root systems in $\tilde{V}_1$ are all simple laced, $\{\varpi_1, \dots, \varpi_{24}\}$ is the dual basis of $x_1,...,x_{24}$. Let $\alpha =\sum_{i=1}^{24} a_i \varpi_i$. Then $a_i= \langle x_i, \alpha\rangle \geq 3/4$. 
Since $\tilde{h}^{\vee}\geq 2$ and $\dim \tilde{V}_1>24$, we have
$$\frac{18}{8}=\langle \alpha,\alpha\rangle\geq \langle \frac{3}{4}\rho,\frac{3}{4}\rho\rangle=\frac{9}{16}\times \frac{\tilde{h}^{\vee}\dim \tilde{V}_1}{12}\gneq \frac{9}{16}\times 4=\frac{9}{4},$$ 
which is a contradiction.  Note that $\rho = \sum_{i=1}^{24} \varpi_i$. 

\subsubsection{Case $N_0=9$, $K_0=10$, $R=2$} 
In this case, we have $P(m,n)= X(\frac{(10m+9n)\alpha}{20})\otimes 
e^{\frac{(-10m+9n)\alpha}{20}}$.  Then all root vectors $e^{x_i}$ are contained in 
$(T^3+T^{12})\cup(T^6+T^{15})$ because  $\mathrm{GCD}(m, 9)\neq 1$ and $m\notin 9\bZ$. 

Let $x$ be a positive root and set $x= k\alpha + x'$ where $x'\in (\bQ \alpha)^\perp$. Then $k=  \frac{(-10m+9n)\alpha}{20}$ for some $m,n\in bZ$. 
Since $\wt(\frac{(-10m+9n)\alpha}{20})\leq 1$, we also know $|-10m+9n| \lneq 20$. 
Therefore, the possible choices of $k=\frac{-10m+9n}{20}\gneq 0$  
are $\frac{6}{20},\frac{15}{20}$, and  $\frac{12}{20}$.
For $k=\frac{6}{20},\frac{15}{20}$, $e^{x}$ is in $T^3+T^{12}$. For $k=\frac{12}{20}$, $e^{x}$  is  in $T^6+T^{15}$.
Suppose $\tilde{V}_1$ is not abelian and let $\theta=\sum a_i x_i$ be a highest root of an irreducible component. Then  $\sum a_i \leq 2$ because $\frac{15}{20} \geq \frac{6}{20} \sum a_i$. Therefore, the dual Coxeter number is $2$ or $3$ and $\tilde{V}_1$ has the type  $(A_q)^{\oplus 24/q}$ with $q=1$ or $2$.

Let $\{x_1,...,x_{24}\}$ be a set of simple roots. By the discussion above, 
$\langle x_i,\alpha \rangle\geq \frac{6}{20} \times \frac{20}{9}=\frac{2}{3}$. 
Using the similar arguments as in Section \ref{caseN0=8},  
$$
\frac{20}{9}=\langle \alpha,\alpha\rangle\geq \langle \frac{2\rho}{3},\frac{2\rho}{3}\rangle=\frac{4}{9} \langle \rho,\rho \rangle=\frac{4}{9}\times \frac{24}{q} \times \frac{(q+1)q(q+2)}{12}=\frac{8(q+1)(q+2)}{9}\geq \frac{48}{9},
$$ 
which is a contradiction, where $\rho$ denotes a Weyl vector for $A_q^{\oplus 24/q}$.

\subsubsection{Case $N_0=15$, $K_0=16$, $R=2$} 
In this case, since $(n,15)\not=1$ and 
$\tilde{\CH}\subseteq T^{0,0}+T^{15,0}$, 
all root vectors $e^{x_i}$ are in the union of $T^{n,0}+T^{n+15,0}$ with $n=3,6,9,12,5,10$. 
Since $\{3m+5n \mid m=1,2,3,4, n=1,2\}\cap \{3,6,9,12,5,10\}=\emptyset$, 
a sum $\delta+\mu$ of a root $\delta$ for $(\sum_t T^{5t,0}+T^{5t+15,0})_1$ and a root 
$\mu$ for $(\sum T^{3s,0}+T^{3s+15,0})_1$ is not a root, that is, they are orthogonal. 
As we did in the previous subsection, we consider the possible choice for $(m,n)$ with $\frac{-16n+15m}{32}>0$ and $n=3,6,9,12,5,10$. 
Let $x= \frac{s}{32}\alpha + x'$ be a positive root.  Then $\langle \frac{s\alpha}{32},\frac{s\alpha}{32}\rangle=\frac{s^2}{32\times 15}\leq 2$ and 
we have $|-16n+15m|\leq 30$. Therefore, the possible values for $\frac{-16n+15m}{32}$ associated  with positive roots are 
$$\begin{array}{ccc}
T^n+T^{n+15} & \frac{-16n+15m}{32} &\langle \alpha,\frac{(-16n+15m)\alpha}{32}\rangle \cr 
T^9+T^{24} &  6/32 & 2/5  \cr  
T^3+T^{18}& 12/32 & 4/5 \cr  
T^{12}+T^{18} & 18/32 & 6/5 \cr  
T^6+T^{21} & 24/32 & 8/5 \cr  
T^5+T^{20} & 10/32 & 2/3 \cr
T^{10}+T^{25} & 20/32 & 4/3 
\end{array}$$ 
From the above, the Coxeter number of $\tilde{V}_1$ is 
$\leq 5$ and so the possible structure of 
$\tilde{V}_1\cong A_1^{24}, A_2^{12}, A_3^{8}, A_4^{6}$, $(A_q)^{\oplus 24/q}$ with $q=1,2,3,4$. Let 
$\rho$ be a Weyl-vector of $(A_q)^{\oplus 24/q}$. 
Since $\langle x_i,\alpha\rangle\geq 2/5$,  we have 
$$ 
\frac{32}{15}=\langle \alpha,\alpha\rangle\geq 
\langle \frac{2}{5}\rho,\frac{2}{5}\rho\rangle 
=\frac{4}{25}\times\frac{24}{q}\times \frac{q(q+1)(q+2)}{12}
=\frac{8(q+1)(q+2)}{25}, $$
by the similar arguments as in Section \ref{caseN0=8}. 
On the other hand, 
$\frac{8(q+1)(q+2)}{25} > \frac{32}{15}$ when $q\geq 2$. 
Therefore, $q=1$ is the only possible solution and 
$\tilde{V}_1\cong (A_1)^{\oplus 24}$.  
In this case, the set of positive roots forms a basis of the root lattice and the Weyl vector 
$\rho=\frac{1}{2} \sum_{i=1}^{24} x_i$, where 
$\{x_1, \dots, x_{24}\}$ is a set of positive roots.

Note that $\dim T^3_1=\dim T^{21}_1=\dim T^9_1=\dim T^{27}_1$ since 
$(3,30)=(21,30)=(9,30)=(27,30)$. 
Then the number of positive root vectors in $(T^9+T^{24})\cup (T^{6}+T^{21})$ and that in $(T^3+T^{18})\cup (T^{12}+T^{27})$
are the same, say $j$.  
Then we have 
$$\frac{32}{15}=\langle \alpha,\alpha\rangle \geq \langle \alpha,\frac{2}5 \rho \rangle=\frac{1}5\sum_{i=1}^{24} \langle \alpha,x_i\rangle\geq \frac{1}5 \left( \frac{2j}{5}+\frac{4j}{5}+(24-2j)\frac{2}{3}\right) =\frac{240-2j}{5\times 15} \geq \frac{216}{75},$$ 
since $j\leq 12$. It is a contradiction and this completes the proof of the main theorem. 
\qed 

\medskip

Since $\rank(V^{[N_0\alpha]}_1) =\rank(V^{[\alpha]}_1)$, we also have the following corollary. 
\begin{cry}
Suppose $N_0\neq 1$. Then $V^{[N_0\alpha]}$ is isomorphic to a Niemeier lattice VOA. 
\end{cry}

\newpage

\appendix 
\section{Some properties of Conway group $Co_0$}
In this appendix, we recall some properties of the group $Co_0$. We use the notation as in \cite[Table 1]{HaLa}.

\begin{thm}\label{thm:A1}
The list of conjugacy classes of Conway group $Co_0$ with nonzero fixed subspace \\
\begin{tabular}{rcc}
class&frame shape&fixed $\dim$\\
\hline
$2A$&$1^{8}2^{8}$&16\\
$-2A$&$2^{16}/1^{8}$&8\\
$2C$&$2^{12}$&12 \\
\hline
$3B$&$1^{6}3^{6}$&12\\
$3C$&$3^{9}/1^{3}$&6\\
$3D$&$3^{8}$&8 \\
\hline
$-4A$&$1^{8}4^{8}/2^{8}$&8
\\
$4B$&$4^{8}/2^{4}$&4
\\
$4C$&$1^{4}2^{2}4^{4}$&10
\\
$-4C$&$2^{6}4^{4}/1^{4}$&6
\\
$4D$&$2^{4}4^{4}$&8
\\
$4F$&$4^{6}$&6 \\
\hline
$5B$&$1^{4}5^{4}$&8
\\
$5C$&$5^{5}/1^{1}$&4 \\
\hline
$6C$&$1^{4}2^{1}6^{5}/3^{4}$&6
\\
$-6C$&$2^{5}3^{4}6^{1}/1^{4}$&6
\\
$-6D$&$1^{5}3^{1}6^{4}/2^{4}$&6
\\
$6E$&$1^{2}2^{2}3^{2}6^{2}$&8
\\
$-6E$&$2^{4}6^{4}/1^{2}3^{2}$&4
\\
$6F$&$3^{3}6^{3}/1^{1}2^{1}$&4
\\
$-6F$&$1^{1}6^{6}/2^{2}3^{3}$&2
\\
$6G$&$2^{3}6^{3}$&6
\\
$6I$&$6^{4}$&4 \\
\hline
$7B$&$1^{3}7^{3}$&6 \\
\hline
$8B$&$2^{4}8^{4}/4^{4}$&4
\\
$-8C$&$1^{4}8^{4}/2^{2}4^{2}$&4
\\
$8D$&$8^{4}/4^{2}$&2
\\
$8E$&$1^{2}2^{1}4^{1}8^{2}$&6
\\
$-8E$&$2^{3}4^{1}8^{2}/1^{2}$&4
\\
$8F$&$4^{2}8^{2}$&4 \\
\hline
$9B$&$9^{3}/3^{1}$&2
\\
$9C$&$1^{3}9^{3}/3^{2}$&4 \\
\hline
$10D$&$1^{2}2^{1}10^{3}/5^{2}$&4
\\
$-10D$&$2^{3}5^{2}10^{1}/1^{2}$&4
\\
$-10E$&$1^{3}5^{1}10^{2}/2^{2}$&4
\\
$10F$&$2^{2}10^{2}$&4 \\
\hline

\end{tabular}\qquad \qquad 
\begin{tabular}{rcc}
class&frame shape&fixed $\dim$\\
\hline
$11A$&$1^{2}11^{2}$&4 \\
\hline
$-12D$&$2^{1}3^{3}12^{3}/1^{1}4^{1}6^{3}$&2
\\
$-12E$&$1^{2}3^{2}4^{2}12^{2}/2^{2}6^{2}$&4
\\
$12G$&$4^{2}12^{2}/2^{1}6^{1}$&2
\\
$12H$&$2^{3}6^{1}12^{2}/1^{1}3^{1}4^{2}$&2
\\
$-12H$&$1^{1}2^{2}3^{1}12^{2}/4^{2}$&4
\\
$12I$&$1^{2}4^{1}6^{2}12^{1}/3^{2}$&4
\\
$-12I$&$2^{2}3^{2}4^{1}12^{1}/1^{2}$&4
\\
$12J$&$2^{1}4^{1}6^{1}12^{1}$&4
\\
$-12K$&$1^{3}12^{3}/2^{1}3^{1}4^{1}6^{1}$&2
\\
$12M$&$12^{2}$&2  \\
\hline
$14B$&$1^{1}2^{1}7^{1}14^{1}$&4
\\
$-14B$&$2^{2}14^{2}/1^{1}7^{1}$&2  \\
\hline
$15D$&$1^{1}3^{1}5^{1}15^{1}$&4
\\
$15E$&$1^{2}15^{2}/3^{1}5^{1}$&2 \\
\hline
$16A$&$2^{2}16^{2}/4^{1}8^{1}$&2
\\
$-16B$&$1^{2}16^{2}/2^{1}8^{1}$&2  \\
\hline
$-18B$&$1^{2}9^{1}18^{1}/2^{1}3^{1}$&2
\\
$18C$&$1^{1}2^{1}18^{2}/6^{1}9^{1}$&2
\\
$-18C$&$2^{2}9^{1}18^{1}/1^{1}6^{1}$&2 \\
\hline
20B&$4^{1}20^{1}$&2
\\
$20C$&$1^{1}2^{1}10^{1}20^{1}/4^{1}5^{1}$&2
\\
$-20C$&$2^{2}5^{1}20^{1}/1^{1}4^{1}$&2  \\
\hline
$21C$&$3^{1}21^{1}$&2  \\
\hline
$22A$&$2^{1}22^{1}$&2
\\
\hline
$23A$&$1^{1}23^{1}$&2
\\
$B**$&$1^{1}23^{1}$&2  \\
\hline
$24E$&$2^{1}6^{1}8^{1}24^{1}/4^{1}12^{1}$&2
\\
$24F$&$1^{1}4^{1}6^{1}24^{1}/3^{1}8^{1}$&2
\\
$-24F$&$2^{1}3^{1}4^{1}24^{1}/1^{1}8^{1}$&2  \\
\hline
$-28A$&$1^{1}4^{1}7^{1}28^{1}/2^{1}14^{1}$&2 \\
\hline
$30D$&$1^{1}6^{1}10^{1}15^{1}/3^{1}5^{1}$&2
\\
$-30D$&$2^{1}3^{1}5^{1}30^{1}/1^{1}15^{1}$&2
\\
$-30E$&$2^{1}3^{1}5^{1}30^{1}/6^{1}10^{1}$&2\\
\hline
\end{tabular}
\end{thm}

As a group theoretical knowledge, we can get the following information 
from the frame shapes:

\begin{remark}\label{frameshape}
\begin{enumerate}
\item If $g$ has a frame shape $g=\prod_{n\in \bN}n^{m_n}$, then 
the dimension of fixed point subspace is given by $\dim \bC\Lambda^{<g>}=\sum_n m_n$. 

\item Let $\chi$ be an irreducible character of $Co_0$ of degree 24. 
If $g\in Co_0$ has a frame shape $\prod_{n\in \bN} n^{m_n}$ 
with $m_n\in \bZ$, then $\chi(g)=m_1$.  

\item Except for $23A$ and $23B$, frame shape determines a unique conjugacy class. Therefore, from a frame shape of $g$, 
it is easy to see conjugacy classes of powers of $g$. 
\end{enumerate}
\end{remark}

For our purpose, we will only use  elements with $\dim \bC\Lambda^{<g>}\geq 4$. 
Using a computer program MAGMA, we have the following list for 
abelian subgroups of rank $2$.

\begin{thm}\label{Conway}
The following is the list of abelian subgroups of rank $2$
which fix at least $4$-dimensional subspace. In the table,  $\#$ denotes the number of conjugacy classes 
but we won't use this number.

$\dim =4$ (total $\#=58$)

\begin{center}
\begin{tabular}{ll}
Structure&\#\\
\hline
$\bZ_{2}\times \bZ_{2}$&3\\
$\bZ_{2}\times \bZ_{4}$&18\\
$\bZ_{3}\times \bZ_{3}$&2\\
$\bZ_{2}\times \bZ_{6}$&8\\
$\bZ_{4}\times \bZ_{4}$&10\\
\end{tabular}\qquad \qquad 
\begin{tabular}{ll}
Structure&\#\\
\hline
$\bZ_{2}\times \bZ_{8}$&8\\
$\bZ_{3}\times \bZ_{6}$&6\\
$\bZ_{2}\times \bZ_{12}$&1\\
$\bZ_{5}\times \bZ_{5}$&1\\
$\bZ_{3}\times \bZ_{9}$&1\\
\end{tabular}
\end{center} 

$\dim =6$ (total $\#=15$)
\begin{center}
\begin{tabular}{lll}
Structure&$\#$&Conjugacy class of elements\\
\hline
$\bZ_{2}\times \bZ_{2}$&1&$\{1A,2A,-2A,2C\}$\\
$\bZ_{2}\times \bZ_{2}$&3&$\{1A,2C,2C,2C\}$\\
$\bZ_{2}\times \bZ_{4}$&1&$\{1A,2A,2A,-2A,4C,4C,4C,4C\}$\\
$\bZ_{2}\times \bZ_{4}$&2&$\{1A,2A,2C,2C,4C,4C,4C,4C\}$\\
$\bZ_{2}\times \bZ_{4}$&1&$\{1A,2A,2A,2C,4C,4C,4D,4D\}$\\
$\bZ_{2}\times \bZ_{4}$&1&$\{1A,2A,2A,2A,4C,4C,-4C,-4C\}$\\
$\bZ_{2}\times \bZ_{4}$&1&$\{1A,2A,2A,2A,4D,4D,4D,4D\}$\\
$\bZ_{3}\times \bZ_{3}$&1&$\{1A,3B,3B,3B,3B,3B,3B,3C,3C\}$\\
$\bZ_{2}\times \bZ_{6}$&2&$\{1A,2A,2A,2A,3B,3B,6E,6E,6E,6E,6E,6E\}$\\
$\bZ_{4}\times \bZ_{4}$&2&$\{1A,2A,2A,2A,4C,4C,4C,4C,4C,4C,4C,4C,4C,4C,4C,4C\}$\\
\end{tabular}
\end{center}

$\dim =8$  ($\#=6$)
\begin{center}
\begin{tabular}{lll}
Structure&$\#$&Conjugacy class of elements\\
\hline
$\bZ_{2}\times \bZ_{2}$&2&$\{1A,2A,2C,2C\}$\\
$\bZ_{2}\times \bZ_{2}$&1&$\{1A,2A,2A,-2A\}$\\
$\bZ_{2}\times \bZ_{4}$&1&$\{1A,2A,2A,2A,4C,4C,4C,4C\}$\\
$\bZ_{2}\times \bZ_{4}$&1&$\{1A,2A,2A,-2A,-4A,-4A,-4A,-4A\}$\\
$\bZ_{3}\times \bZ_{3}$&1&$\{1A,3B,3B,3B,3B,3B,3B,3B,3B\}$\\
\end{tabular}
\end{center}

$\dim =10$ (total $\#=1$)
\begin{center}
\begin{tabular}{lll}
Structure&$\#$&Conjugacy class of elements\\
\hline
$\bZ_{2}\times \bZ_{2}$&1&$\{1A,2A,2A,2C\}$\\
\end{tabular}
\end{center}

$\dim =12$ (total $\#=1$)
\begin{center}
\begin{tabular}{lll}
Structure&$\#$&Conjugacy class of elements\\
\hline
$\bZ_{2}\times \bZ_{2}$&1&$\{1A,2A,2A,2A\}$\\
\end{tabular}
\end{center}
\end{thm}

\end{document}